\documentclass[11pt]{amsart}
\usepackage{booktabs}
\usepackage{tabularx}
\usepackage{amsmath,amsthm,accents,amssymb, amscd, tikz, tikz-cd}
\usepackage{graphicx,enumerate,stmaryrd}
\usepackage[all,cmtip]{xy}
\usetikzlibrary{decorations.markings}
\tikzstyle{vertex}=[circle, draw, inner sep=0pt, minimum size=6pt]

\usepackage{comment}
\usepackage[margin=1in]{geometry}
\usepackage[bookmarks]{hyperref}
\usepackage{verbatim}

\usepackage[shortlabels]{enumitem}

\theoremstyle{plain}
\newtheorem{thm}{Theorem}[subsection]

\newtheorem{lemma}[thm]{Lemma}

\theoremstyle{definition}
\newtheorem{definition}[thm]{Definition}

\theoremstyle{remark}
\newtheorem{remark}[thm]{Remark}

\newcommand{\Cb}{{\mathbf C}}

\newcommand{\Zb}{{\mathbf Z}}

\newcommand{\ep}{{\varepsilon}}

\numberwithin{equation}{section}
\numberwithin{table}{section}

\begin{document}
 
\author{Nikolay Grantcharov}
\title{Deformation theory of the Double Affine Hecke algebra of type $(C_n^\vee,C_n)$}
\date{June 2026}
\begin{abstract}
We study the double affine Hecke algebra (DAHA) of type $(C_n^\vee,C_n)$ from the perspective of deformation theory. First, we provide a zeros-and-residues realization of this algebra, extending the construction of Ginzburg, Kapranov, and Vasserot to the non-reduced affine root system setting. Specializing the parameters of the DAHA to the base point gives the crossed product of a quantum torus algebra with the finite Weyl group of type $C_n$. We then show that for all $n$, the completed DAHA is the formal universal deformation of this crossed product algebra, extending Oblomkov’s result for $n=1$. Our proof explicitly identifies the completed DAHA with the undeformed crossed product algebra equipped with a formal star product.
\end{abstract}
\maketitle

\section{Introduction}

Affine Hecke algebras were introduced by Iwahori and Matsumoto as deformations of the group algebras of affine Weyl groups. There is a Hecke algebra attached to each finite root system, and its parameters are indexed by orbits of the finite Weyl group on the root system. Their structure and representation theory were developed extensively by Lusztig\cite{Lusztig1989}. Cherednik subsequently introduced double affine Hecke algebras (DAHAs), which may be viewed as the Hecke algebras attached to affine root systems. Their parameters are indexed by orbits of the affine Weyl group on the affine root system. Cherednik used DAHAs to prove the Macdonald conjectures for reduced root systems. Sahi later extended the construction to the non-reduced affine root system of type $(C_n^\vee,C_n)$. This DAHA has five parameters for $n=1$ and six parameters for $n>1$, and it governs the Koornwinder polynomials, which are the most general family of Macdonald polynomials. Sahi used it to prove the remaining Macdonald conjectures in this setting \cite{Sahi1999}.

In this paper, we study the DAHA of type $(C_n^\vee,C_n)$ from the perspective of deformation theory to establish two new results about it.

First, in Theorems~\ref{zeros residues GKV} and~\ref{BEG zeros residues theorem}, we give a zeros-and-residues construction of the DAHA of type $(C_n^\vee,C_n)$, extending the constructions of \cite{GKV} and \cite{BaranovskyEvensGinzburgQuantumToriDAHA} from reduced to non-reduced affine root systems. These realizations describe the DAHA as a subalgebra of rational $q$-difference-reflection operators cut out by explicit vanishing and residue conditions. Such realizations of algebras have proved to be quite versatile. For example, in \cite{GKV}, this result was used to construct degenerate and elliptic analogues of affine and double affine Hecke algebras. More recently, analogous descriptions have appeared for Coulomb branches \cite{SchraderShapiroCoulombResidues, KlyuevResidueCoulombBranches} and, more generally, convolution algebras \cite{CrisanConvolutionAlgebras}. Our result places the non-reduced DAHA in the same framework. In view of the known rank-one case \cite{YoshidaQuantizedCoulombDAHA}, it also suggests a higher-rank Coulomb-branch realization of this DAHA: it would suffice to identify the corresponding Coulomb branch with the same zeros-and-residues subalgebra.

Second, in Theorem~\ref{DAHA universal deformation}, we show that the completed DAHA of type $(C_n^\vee,C_n)$ is the formal universal deformation of the crossed product of a quantum torus algebra with the finite Weyl group of type $C_n$. Previously, Oblomkov proved the result for $n=1$ \cite{Oblomkov2004}. Etingof and Oblomkov subsequently computed the Hochschild cohomology of the crossed product algebra for arbitrary $n$ \cite{EO}. In particular, their computation shows the number of parameters of the DAHA equals the dimension of the second Hochschild cohomology and the infinitesimal deformations are unobstructed since the odd cohomology groups vanish. However, this alone does not imply universality of the deformation since the deformations may a priori be trivial. \footnote{It was suggested to us by Etingof that one could follow the strategy of \cite{EtingofCherednikHeckeVarieties} and \cite{Vitanov2019} to check the deformations are non-trivial by taking completions near the codimension-two fixed-point loci of the $W$-action.}

We prove universality of the deformation by explicitly associating to each of the DAHA parameters a Hochschild $2$-cocycle and showing that the corresponding classes form a basis of the second Hochschild cohomology. In other words, we compute the Kodaira–Spencer morphism in terms of the DAHA parameters and show it is an isomorphism. The cocycles are constructed by comparing the PBW basis of the deformed DAHA, arising from Noumi’s polynomial representation \cite{Noumi1995}, with the standard PBW basis of its specialization. The same construction thus applies more generally to DAHAs admitting a PBW basis compatible with specialization. The main difficulty, however, is showing that the resulting cocycles span the entire second Hochschild cohomology.

\medskip

\noindent\textbf{Acknowledgements.}
The author thanks Victor Ginzburg for suggesting this problem and for many productive discussions. Additionally, the author thanks Pavel Etingof, Daniil Klyuev, and Sarah Witherspoon for helpful comments and conversations.

\section{Residue Construction of $CC_n^\vee$ DAHA}
In \cite{GKV}, Ginzburg, Kapranov, and Vasserot give a uniform construction of affine and double affine Hecke algebras from root data using explicit vanishing and residue conditions on coefficients. More precisely, \cite[Theorems~1.7 and 1.8]{GKV} show that finite root systems give rise to affine Hecke algebras, while reduced affine root systems give rise to double affine Hecke algebras. Here, we extend this construction to the non-reduced affine root system of type $(C_n^\vee,C_n)$ and recover Sahi’s six-parameter DAHA \cite{Sahi1999}.
\subsection{The $(C_n^\vee,C_n)$ root system}
First we recall the root system of type $C_n$. Fix $V=\oplus_{i=1}^{n}\mathbf{R}\ep_i$ and a standard inner product $\langle \ep_i,\ep_j\rangle =\delta_{ij}$. The roots, resp. coroots of type $C_n$ are 
\begin{align*}
R&=\{\pm\ep_i\pm\ep_j:i\neq j\}\cup\{\pm2\ep_i: i=1,\dots, n\}\subset V\\
R^\vee&:=\{\alpha^\vee:=\frac{2\alpha}{\langle \alpha,\alpha\rangle}:\alpha\in R\} = \{\pm\ep_i\pm\ep_j:i\neq j\}\cup\{\pm\ep_i:i=1,\dots, n\}
\end{align*}

The simple roots, resp. coroots are:
\begin{align*}
\Delta(R):=\{\alpha_1&=\ep_1-\ep_2, \dots, \alpha_{n-1}=\ep_{n-1}-\ep_n, \alpha_n=2\ep_n\}\\
\Delta(R^\vee):=\{\alpha_1^\vee&=\ep_1-\ep_2,\dots, \alpha_{n-1}^\vee=\ep_{n-1}-\ep_n, \alpha_n^\vee=\ep_n\}
\end{align*}

The positive roots are 
$$R^+=\{\ep_i\pm\ep_j:i< j\}\cup\{2\ep_i:i=1,\dots,n\},\;\;\; (R^\vee)^+:=\{\ep_i\pm\ep_j:i<j\}\cup\{\ep_i:i=1,\dots, n\}$$
The negative roots are $R^-=-R^+,(R^\vee)^-:=-(R^\vee)^+$, and in both cases, we have $R=R^+\cup R^-,R^\vee=(R^\vee)^+\cup(R^\vee)^-$.

The finite Weyl group of type \(C_n\) is
\[
W_0(C_n)
=
\left\langle s_1,\dots,s_n \ \middle|\
\begin{array}{ll}
s_i^2=1, & 1\leq i\leq n,\\[2pt]
s_i s_j=s_j s_i, & |i-j|>1,\\[2pt]
s_i s_{i+1}s_i=s_{i+1}s_i s_{i+1}, & 1\leq i\leq n-2,\\[2pt]
s_{n-1}s_n s_{n-1}s_n=s_n s_{n-1}s_n s_{n-1}
\end{array}
\right\rangle.
\]

Finally we define:
\begin{align*}
\text{ The root lattice }Q&:=\oplus_{i=1}^n\mathbf{Z}\alpha_i\\
\text{ The weight lattice } \Lambda &:= \oplus_{i=1}^n\mathbf{Z}\ep_i = Q^\vee
\end{align*}

Next, let us recall the affine root system of type $(C_n^\vee, C_n)$. Let $F$ denote the vector space of all affine linear functions from $V$ to $\mathbf{R}$. It may be identified with 
$$F\simeq V\oplus \mathbf{R}\delta.$$
\begin{definition}
The affine roots of type $C_n, C_n^\vee,(C_n^\vee,C_n)$ are respectively
\begin{align}
S&:=S(C_n):=\{\pm 2\ep_i+k\delta: k\in\mathbf{Z},i=1,\dots, n\}\cup\{\pm\ep_i\pm\ep_j + k\delta: k\in\mathbf{Z}, 1\leq i<j\leq n\}\\
S^\vee&:=S(C_n^\vee):=\{\pm \ep_i+\frac{k}{2}\delta: k\in\mathbf{Z},i=1,\dots, n\}\cup\{\pm\ep_i\pm\ep_j + k\delta: k\in\mathbf{Z}, 1\leq i<j\leq n\}\\
\tilde{S}&:=S(CC_n^\vee):=S\cup S^\vee
\end{align}
\end{definition}

The positive roots of $\tilde{S}$ are 
\begin{align*}\tilde{S}^+:=&\{\alpha+k\delta,\alpha^\vee+\frac{k}{2}\delta:\alpha\in R^+,k\in\mathbf{Z}_{\geq0},\alpha^\vee\in (R^+)^{\vee}\}\bigcup\\
&\{\alpha+k\delta,\alpha^\vee+\frac{k}{2}\delta:\alpha\in R^-,k\in\mathbf{Z}_{>0},\alpha^\vee\in (R^-)^{\vee}\}
\end{align*}
and the negative roots are $\tilde{S}^-=-\tilde{S}^+$.

There are two affine simple roots of type $(C_n^\vee,C_n)$ associated to the two reduced subsystems:
$$\alpha_0:=\delta-2\ep_1,\;\;\; \alpha_0^\vee:=\frac{1}{2}\delta-\ep_1.$$ 
The simple roots of type $(C_n^\vee,C_n)$ are thus 
$$\Delta(\tilde{S})=\Delta(S^\vee):=\{a_0^\vee=-\ep_1+\frac{1}{2}\delta,\; a_i^\vee=\ep_i-\ep_{i+1},1\leq i\leq n-1,\; a_n^\vee=\ep_n\}.$$

The corresponding affine Dynkin diagram of type $(C_n^\vee,C_n)$ is 
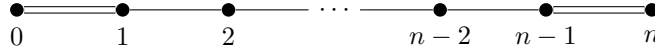
\[
\begin{tikzpicture}[baseline=-0.6ex, scale=1]
  \tikzset{
    vertex/.style={circle, fill=black, inner sep=1.8pt},
    vlabel/.style={font=\small}
  }

  \node[vertex] (0) at (0,0) {};
  \node[vertex] (1) at (1.4,0) {};
  \node[vertex] (2) at (2.8,0) {};
  \node[vlabel] (dots) at (4.2,0) {\(\cdots\)};
  \node[vertex] (n2) at (5.6,0) {};
  \node[vertex] (n1) at (7.0,0) {};
  \node[vertex] (n) at (8.4,0) {};

  \node[vlabel] at (0,-0.35) {\(0\)};
  \node[vlabel] at (1.4,-0.35) {\(1\)};
  \node[vlabel] at (2.8,-0.35) {\(2\)};
  \node[vlabel] at (5.6,-0.35) {\(n-2\)};
  \node[vlabel] at (7.0,-0.35) {\(n-1\)};
  \node[vlabel] at (8.4,-0.35) {\(n\)};

  \draw[double distance=2pt] (0) -- (1);
  \draw (1) -- (2);
  \draw (2) -- (dots);
  \draw (dots) -- (n2);
  \draw (n2) -- (n1);
  \draw[double distance=2pt] (n1) -- (n);
\end{tikzpicture}
\]

Let \(s_i:=s_{\alpha_i^\vee}\) denote the corresponding simple reflections.
Then the affine Weyl group of type \((C_n^\vee,C_n)\) may be presented as
\[
W:=W_{\mathrm{aff}}(C_n)
=
\left\langle s_0,s_1,\dots,s_n \ \middle|\
\begin{array}{ll}
s_i^2=1, & 0\leq i\leq n,\\[2pt]
s_i s_j=s_j s_i, & |i-j|>1,\\[2pt]
s_i s_{i+1}s_i=s_{i+1}s_i s_{i+1}, & 1\leq i\leq n-2,\\[2pt]
s_is_{i+1}s_is_{i+1}=s_{i+1}s_is_{i+1}s_i, & i=0,n-1 \\[2pt]
\end{array}
\right\rangle.
\]
\noindent The Coxeter groups associated to the affine root systems of type $\tilde{C}_n$, $\tilde{C}_n^\vee$ and $(C_n^\vee,C_n)$ are all the same, hence our notation $W:=W_{\mathrm{aff}}(C_n)$. We write $W_0$ for the underlying finite Weyl group. Thus
$$W=\tau(\bigoplus_{i=1}^n\mathbf{Z}\ep_i)\rtimes W_0$$
 where the $\tau(n_i\ep_i)=\tau(\ep_i)^{n_i}$ are translations by $n_i\in\mathbf{Z}$ units in the $\ep_i$ direction. Explicitly,
\begin{equation}\label{affine translation}
\tau(-\ep_i) = (s_{i-1}\cdots s_0)(s_1\cdots s_n)(s_{n-1}\cdots s_i)
\end{equation}
There is a natural action of $W$ on $V$ given by
\begin{align*} 
&s_0\cdot v=(-v_1-1,v_2,\dots, v_n),\\
&s_i\cdot v = (v_1,\dots, v_{i+1},v_i,\dots, v_n),\;\text{ for }1\leq i\leq n-1,\\
&s_n\cdot v=(v_1,\dots, v_{n-1},-v_n).
\end{align*}
This action naturally extends to an action of $W$ on $F$ given by
$$s_i(v+r\delta):= s_iv+r\delta, i\neq 0,\;\;\; s_0(v+r\delta)=(-v_1,v_2,\dots, v_n)+(r-v_1)\delta$$
We find that there are four $W$-orbits on $\tilde{S}$ when $n=1$ and five $W$-orbits on $\tilde{S}$ when $n>1$. To each orbit we associate a parameter as defined in the following table:

\begin{table}[ht]\label{CCn parameters}
\centering
\renewcommand{\arraystretch}{1.2}
\begin{tabular}{ccl}
Parameter & Representative & \(W\)-orbit \\[3pt]

\(u_n\) & \(\alpha_n^\vee=\varepsilon_n\)
& \(\{\pm\varepsilon_i+r\delta\}\) \\

\(t_n\) & \(\alpha_n=2\varepsilon_n\)
& \(\{\pm2\varepsilon_i+2r\delta\}\) \\

\(u_0\) & \(\alpha_0^\vee=\frac{\delta}{2}-\varepsilon_1\)
& \(\{\pm\varepsilon_i+(r+\frac12)\delta\}\) \\

\(t_0\) & \(\alpha_0=\delta-2\varepsilon_1\)
& \(\{\pm2\varepsilon_i+(2r+1)\delta\}\) \\

\(t\) & \(\alpha_i=\varepsilon_i-\varepsilon_{i+1}\)
& \(\{\pm\varepsilon_i\pm\varepsilon_j+r\delta\}\)
\end{tabular}
\caption{The parameters attached to the five \(W\)-orbits in
\(\tilde{S}\)}
\label{table:root-orbit-parameters}
\end{table}

We conclude this subsection by observing the natural $W$ action on $V$ determines a natural $W$ action on $\mathbf{Z}[q^{\pm1}][X_1^{\pm1},\dots, X_n^{\pm1}]$ via:
\begin{align*}
s_0f&=f(q^{-1}X_1^{-1},X_2,\dots, X_n)\\
s_i.f&=f(X_1,\dots, X_{i+1},X_i,\dots, X_n)\;\text{for i}\neq 0,n\\
s_n.f&:= f(X_1,\dots, X_{n-1},X_n^{-1})\\
\tau(\ep_i).f&:= T_{q,X_i}(f):= f(X_1,\dots, qX_i,\dots, X_n)
\end{align*}
Thus, we see $\tau(\ep_i)$, the affine translations, correspond to the $i^{th}$ $q$-shift operators acting on Laurent polynomials.

\subsection{(Double) Affine Hecke algebra associated to $(C_n^\vee,C_n)$ root system}
We follow Noumi's labeling of the parameters of the DAHA (see \cite{Noumi1995,Yamaguchi2022}) associated to the affine root system of type $(C_n^\vee,C_n)$. Define the parameters 
 \begin{align*}
     \underline{\kappa}&=(t_0^{1/2},t^{1/2},t_n^{1/2},u_0^{1/2},u_n^{1/2})
 \end{align*}
Consider the ground ring 
\begin{equation}\label{ground ring}
\mathbf{F}:=\mathbf{Z}[q^{\pm 1/2},t^{\pm 1/2},t_0^{\pm1/2},t_n^{\pm1/2},u_0^{\pm1/2},u_n^{\pm1/2}].
\end{equation}

\begin{definition}\label{rational functions on roots}
Let $\Lambda:=\bigoplus_{i=1}^n\mathbf Z\epsilon_i$
be the weight lattice of type \(C_n\). For an affine weight $\mu=\sum_{i=1}^n\lambda_i\ep_i+k\frac{\delta}{2}\in \Lambda+\frac{1}{2}\mathbf{Z}\delta$, we denote
\begin{equation}\label{roots to polynomial}
e^\mu:=q^{k/2}X_1^{\lambda_1}\cdots X_n^{\lambda_n}.
\end{equation}
\end{definition}

Let $T$ be the torus whose character lattice is $\Lambda$. Observe \footnote{We do not treat $e^\lambda$ as the functional which sends $\mu\in S$ to $q^{\langle\lambda,\mu\rangle}$ as in \cite{GKV} because we want to work formally, without ``evaluations.''}, $e^\lambda e^\mu=e^{\lambda+\mu}$ and 
\begin{equation}\label{e^alpha}
e^{\alpha_i} = \begin{cases} X_i/X_{i+1}&\text{ if }1\leq i\leq n-1\\  X_n^2 &\text{ if }i=n\\ qX_1^{-2}&\text{ if }i=0\end{cases},\quad e^{\alpha_i^\vee} = \begin{cases} X_i/X_{i+1}&\text{ if }1\leq i\leq n-1\\  X_n &\text{ if }i=n\\ q^{\frac{1}{2}}X_1^{-1}&\text{ if }i=0\end{cases}\end{equation}

\begin{definition}
    An affine root $\lambda\in \tilde{S}$ is \textit{real} if $\lambda=\lambda_0+k\delta$ where $\lambda_0\in V\setminus \{0\}$. Denote the set of positive real roots by $\tilde{S}_{Real}^+$. Similarly, denote the real roots of $S$ by $S_{\mathrm{Real}}$.
\end{definition}
\begin{definition}
   Define
\[
\mathbf F[T]
=
\mathbf F[e^\lambda:\lambda\in\Lambda]\simeq \mathbf{F}[X_1^{\pm1},\dots,X_n^{\pm1}].
\]
and
\[
\mathbf F[T]_{\mathrm{loc}}
:=
\mathbf F[T]\left[(1-e^\alpha)^{-1}:
\alpha\in S_{\mathrm{Real}}\right].
\]
\end{definition}
\noindent The localized algebra $\mathbf{F}[T]_{\mathrm{loc}}$ replaces the role of $\mathbf{F}(T):=\mathrm{Frac}(\mathbf{F}[T])$ in \cite{GKV}. The notation $\mathbf F[T]_{\mathrm{loc}}$ is consistent with the localized abelian Coulomb-branch algebras appearing in the
abelianization construction of \cite{BFN}. We work with the integral version because this behaves better under base change to characteristic $p$.

 \begin{definition}
 Define the affine Hecke algebra $H_{\underline{t}}(C_n)$ of type ${C}_n$ to be the algebra generated over $\mathbf{F}$ by $T_0^{\pm1},T_1^{\pm1},\dots, T_n^{\pm1}$, subject to the type $\tilde{C}_n$ braid relations and the quadratic relations
 $$T_i-T_i^{-1}=t_i^{1/2}-t_i^{-1/2}$$
  where $t_1=t_2=\dots=t_{n-1}=t$. 
  \end{definition}

The elements $T_1,\dots, T_n$ generate the finite Hecke algebra $H_0$ of type $C_n$. The analogs of the translations $\tau(\ep_i)$ are the elements 
\begin{equation}\label{Y_i def}
Y_i:=(T_i\cdots T_{n-1})(T_n\cdots T_0)(T_1^{-1}\cdots T_{i-1}^{-1}), \;\;i=1,\cdots, n
\end{equation}
Lusztig \cite{Lusztig1989} showed the $Y_i$ pairwise commute and generate a subalgebra $\mathbf{F}[Y_1^{\pm1},\dots, Y_n^{\pm1}]$ of $H_{\underline{t}}(C_n)$ such that multiplication gives an isomorphism 
$$H_0\otimes \mathbf{F}[Y_1^{\pm1},\dots, Y_n^{\pm1}]\xrightarrow{\sim} H_{\underline{t}}(C_n).$$

Now, there is also the standard polynomial representation, due to Noumi \cite{Noumi1995}, 
\begin{equation}\label{Noumi}
\pi:H_{\underline{t}}(C_n)\rightarrow \text{End}_{\mathbf{F}}(\mathbf{F}[X_1^{\pm1},\dots,X_n^{\pm1}]),
\end{equation}
\begin{align}\label{Noumi polynomial}
T_i^{\pm1}&\mapsto t_i^{\pm \frac{1}{2}}+t_i^{-\frac{1}{2}}\frac{1-t_iX_i/X_{i+1}}{1-X_i/X_{i+1}}(s_i-1),\;\; i=1,\dots, n-1\\
T_0^{\pm1}&\mapsto t_{0}^{\pm \frac{1}{2}}+t_0^{-\frac{1}{2}}\frac{(1-q^{\frac{1}{2}}u_0^{\frac{1}{2}}t_0^{\frac{1}{2}}X_1^{-1})(1+q^{\frac{1}{2}}u_0^{-\frac{1}{2}}t_0^{\frac{1}{2}}X_1^{-1})}{1-qX_1^{-2}}(s_0-1)\\
T_n^{\pm1}&\mapsto t_n^{\pm \frac{1}{2}}+t_n^{-\frac{1}{2}}\frac{(1-u_n^{\frac{1}{2}}t_n^{\frac{1}{2}}X_n)(1+u_n^{-\frac{1}{2}}t_n^{\frac{1}{2}}X_n)}{1-X_n^2}(s_n-1)
\end{align}
\begin{remark}
Stokman gives an equivalent realization of the same
polynomial representation. Namely, in \cite[Theorem 9.2.3]{Stokman2000}, the
affine Hecke algebra embeds into \(\mathbf{F}[T]_{\mathrm{loc}}\rtimes W\), and the image of
\(T_i\) is written in terms of parameters \(q_a\) attached to \(W\)-orbits of
affine roots. The transition between Stokman's and Noumi's parameters is: 
\[
(q_0,q_{2\alpha_0},q_n,q_{2\alpha_n},q_i)
=
(t_0^{1/2},u_0^{1/2},t_n^{1/2},u_n^{1/2},t_i^{1/2}),
\qquad 1\leq i\leq n-1.
\]
For reduced simple roots \(1\leq i\leq n-1\), this specializes to the usual
Demazure--Lusztig coefficient. Thus Stokman's embedding and Noumi's polynomial
representation are the same representation, written respectively inside the
rational crossed product \(\mathbf{F}[T]_{\mathrm{loc}}\rtimes W\) and as operators on
\(\mathbf F[X_1^{\pm1},\dots,X_n^{\pm1}]\).
\end{remark}

Finally, we give the definitions of the double affine Hecke algebra $H_{q,\underline{\kappa}}(CC_n^\vee)$ associated to the $(C_n^\vee,C_n)$-root system $\tilde{S}$ using the original definition of \cite{Sahi1999}.

\begin{definition}\label{generators and relations DAHA}
Define $H_{q,\underline{\kappa}}(CC_n^\vee)$ to be the $\mathbf{F}$-algebra generated by $T_i^{\pm1},i=0,1,\dots, n$ and commuting elements $X_i^{\pm1},i=1,\dots, n$, subject to the following six relations: 
 We use the notation 
$$T_i\sim t_i \text{ means } T_i-T_i^{-1}=t_i^{1/2}-t_i^{-1/2}.$$
 \begin{enumerate}
 \item $T_i\sim t_i$
 \item The \(T_i\)'s satisfy the affine \(\widetilde C_n\) braid relations:
    \[
    T_iT_j=T_jT_i,\qquad |i-j|>1,
    \]
    \[
    T_iT_{i+1}T_i=T_{i+1}T_iT_{i+1},
    \qquad 1\leq i\leq n-2,
    \]
    and
    \[
    T_iT_{i+1}T_iT_{i+1}
    =
    T_{i+1}T_iT_{i+1}T_i,
    \qquad i=0,n-1.
    \]

 \item $T_iX_j=X_jT_i$ if $|i-j|>1$ or $(i,j)=(n,n-1)$
 \item $T_iX_{i}=X_{i+1}T_i^{-1}$ for $i=1,2,\dots, n-1$
 \item $X_n^{-1}T_n^{-1}\sim u_n$
 \item $q^{-1/2}T_0^{-1}X_1\sim u_0$
\end{enumerate}

\end{definition}

The affine Hecke algebra \(H_{\underline t}(C_n)\) is the subalgebra of $H_{q,\underline{\kappa}}(CC_n^\vee)$ generated by \(T_0,\dots,T_n\). The elements \(X_1^{\pm1},\dots,X_n^{\pm1}\) generate the
Laurent polynomial subalgebra. 

Noumi's polynomial representation gives an action of \(H_{\underline t}(C_n)\)
on $\mathbf F[X_1^{\pm1},\dots,X_n^{\pm1}]$
by the formulas in \eqref{Noumi}. Extending this action by letting the generator
\(X_i\) act by multiplication by \(X_i\), one obtains a representation of
\(H_{q,\underline{\kappa}}(CC_n^\vee)\). By \cite[Theorem 3.2]{Sahi1999}, this
representation is faithful (in fact, simple). Thus we may equivalently regard
\(H_{q,\underline{\kappa}}(CC_n^\vee)\) as the subalgebra
\[
H_{q,\underline{\kappa}}(CC_n^\vee)
\subset
\operatorname{End}_{\mathbf F}
\bigl(\mathbf F[X_1^{\pm1},\dots,X_n^{\pm1}]\bigr)
\]
generated by Noumi's operators \(T_0,\dots,T_n\) together with multiplication by
\(X_1^{\pm1},\dots,X_n^{\pm1}\).

\subsection{Zeros and Residues construction of $H_{q,\underline{\kappa}}(CC_n^\vee)$}
We now explain an equivalent definition of $H_{q,\underline{\kappa}}(CC_n^\vee)$ using zeros and residues, in the spirit of \cite{GKV}.

Following Yamaguchi's notation \cite{Yamaguchi2022}, for a root $\alpha\in S_{\mathrm{Real}}$ set
\begin{equation}\label{u_alpha t_alpha def}
u_\alpha:=
\begin{cases}
1, & \alpha\in W\cdot\alpha_i,\quad 1\leq i\leq n-1,\\
u_0, & \alpha\in W\cdot\alpha_0,\\
u_n, & \alpha\in W\cdot\alpha_n,
\end{cases}
\qquad
t_\alpha:=
\begin{cases}
t, & \alpha\in W\cdot\alpha_i,\quad 1\leq i\leq n-1,\\
t_0, & \alpha\in W\cdot\alpha_0,\\
t_n, & \alpha\in W\cdot\alpha_n.
\end{cases}
\end{equation}

For \(0\leq i\leq n\) and simple root $\alpha_i\in S$, define $u_i:=u_{\alpha_i}$ and $t_i:=t_{\alpha_i}$. Further define
\begin{align*}
c_i(z)&:=
t_i^{-\frac{1}{2}}
\frac{
(1-u_i^{\frac{1}{2}}t_i^{\frac{1}{2}}z^{\frac{1}{2}})
(1+u_i^{-\frac{1}{2}}t_i^{\frac{1}{2}}z^{\frac{1}{2}})
}
{1-z}\\
d_i(z)&:=
t_i^{\frac{1}{2}}-c_i(z)
=
\frac{
(t_i^{\frac{1}{2}}-t_i^{-\frac{1}{2}})
+
(u_i^{\frac{1}{2}}-u_i^{-\frac{1}{2}})z^{\frac{1}{2}}
}
{1-z}.
\end{align*}
Finally, define
\begin{equation}\label{beta_i basis}
\beta_i:=\pi(T_i)
=
c_i(e^{\alpha_i})[s_i]+d_i(e^{\alpha_i})[1]\in\mathbf{F}[T]_{\mathrm{loc}}\rtimes W.
\end{equation}
where $e^{\alpha_i}$ are as defined in Equation \ref{e^alpha}. Notice, for the reduced roots, i.e., when \(1\leq i\leq n-1\), the numerator of $c_i(z)$ simplifies as
\[
(1-t^{1/2}z^{1/2})(1+t^{1/2}z^{1/2})=1-tz.
\]
Thus $\beta_i$ simplifies to the familiar Demazure-Lusztig operators:
\[
\beta_i
=
\frac{t^{-\frac{1}{2}}-t^{\frac{1}{2}}X_i/X_{i+1}}{1-X_i/X_{i+1}}[s_i]
+
\frac{t^{\frac{1}{2}}-t^{-\frac{1}{2}}}{1-X_i/X_{i+1}}[1],\;\; 1\leq i\leq n-1.
\]
Finally, for a root \(\alpha\in\widetilde S\), denote by
\[
T_{\alpha,a}:=\{x\in T:e^\alpha(x)=a\},
\qquad
T_\alpha:=T_{\alpha,1}.
\]

\begin{definition}\label{zeros and residue of DAHA}
Let \(\tilde{S}\) be the type \((C_n^\vee,C_n)\) root system and let $S:=S(C_n)$ be the reduced subsystem of affine type $C_n$. Recall Definition
\ref{rational functions on roots} of the space \(\mathbf{F}[T]_{\mathrm{loc}}\) of rational functions on \(T\). Let \(\mathbf{H}_{q,\underline{\kappa}}\) denote the
\(\mathbf F\)-linear subspace of
\[
\mathbf{F}[T]_{\mathrm{loc}}\rtimes W=\bigoplus_{w\in W}\mathbf{F}[T]_{\mathrm{loc}}[w]
\]
consisting of elements \(f=\sum_{w\in W}f_w[w]\)
such that:
\begin{enumerate}
\item Each coefficient \(f_w\) has at most first-order poles along the divisors
\(T_\alpha\), \(\alpha\in S_{\mathrm{Real}}^+\).

\item For each \(w\in W\) and \(\alpha\in S_{\mathrm{Real}}^+\), we have
\[
\operatorname{Res}_{T_\alpha}(f_w)
+
\operatorname{Res}_{T_\alpha}(f_{s_\alpha w})
=0.
\]

\item Recall the parameters $(t_\alpha,u_\alpha)$ as in \ref{u_alpha t_alpha def}. Define 
$$D(w):=S_{\mathrm{Real}}^+\cap w(S^-).$$ 
 Then for each $\alpha\in D(w)$, the function $f_w$ vanishes on 
\[
\widetilde T_\alpha:=
\begin{cases}
T_{\alpha,t_\alpha^{-1}},
& \alpha/2\notin\widetilde S,\\[6pt]
T_{\frac{\alpha}{2},\,t_\alpha^{-1/2}u_\alpha^{-1/2}}
\cup
T_{\frac{\alpha}{2},\,-t_\alpha^{-1/2}u_\alpha^{1/2}},
& \alpha/2\in\widetilde S.
\end{cases}
\]
We interpret vanishing on a union to mean vanishing on both components.
\end{enumerate}
\end{definition}

\begin{remark} Also we emphasize that conditions 1, 2 and 3 are parameterized by the underlying reduced root system, and the key distinction from the DAHA of type $C_n$ is the non-reducedness condition on the zeros appearing in condition 3. Furthermore, observe that for $\alpha/2\in\tilde{S}$, if we plug in $t_\alpha=t,u_\alpha=1$ into divisor $\tilde{T}_\alpha$, the condition becomes $f_w$ vanishes on $(e^{\alpha/2}-t^{-1/2})(e^{\alpha/2}+t^{-1/2})=e^\alpha-t^{-1}$, which is precisely $T_{\alpha,t^{-1}}$!
\end{remark}
Here is the main theorem of this section:
\begin{thm}\label{zeros residues GKV}
The $\mathbf{F}$-linear space $\mathbf{H}_{q,\underline{\kappa}}$, as defined in Definition \ref{zeros and residue of DAHA}, is an algebra which is isomorphic to the $(C_n^\vee,C_n)$-DAHA $H_{q,\underline{\kappa}}(CC_n^\vee)$, as defined in Definition \ref{generators and relations DAHA}.
\end{thm}
\begin{proof}
We divide the proof into three steps, following a similar outline as in the proof of \cite{GKV}.

\textbf{Step 1}. We show the subspace \(\mathbf{H}_{q,\underline{\kappa}}\) is an
\(\mathbf F\)-algebra. In \cite[Theorem 1.4]{GKV}, it is shown that
conditions (1) and (2) cut out an algebra, which \cite{GKV} denote by
\(\widetilde H\). The key ingredient in this proof is just that \(W\) is a
Coxeter group. To show that further imposing our condition (3) also cuts out
an algebra, we use similar reasoning.

Write
\[
P=\sum_{w\in W}P_w[w],
\qquad
Q=\sum_{w\in W}Q_w[w]
\]
with \(P,Q\in \mathbf{H}_{q,\underline{\kappa}}\). Then
\[
PQ=\sum_{u\in W}F_u[u],
\qquad
F_u=\sum_{wy=u}P_w\,{}^wQ_y.
\]
We show that each summand \(P_w\,{}^wQ_y\) vanishes on the divisor
\(\widetilde T_\alpha\) required by condition (3).

First observe that the divisors \(\widetilde T_\alpha\) are \(W\)-equivariant:
\[
w(\widetilde T_\beta)=\widetilde T_{w\beta}.
\]
Indeed, if \(\beta/2\notin \tilde{S}\), then
\(\widetilde T_\beta=T_{\beta,t^{-1}}\), and this follows immediately from
\(w(e^\beta)=e^{w\beta}\). If \(\beta/2\in \tilde{S}\), then \(\widetilde T_\beta\)
is the zero locus of
\[
(e^{\beta/2}-t_\beta^{-\frac12}u_\beta^{-\frac12})
(e^{\beta/2}+t_\beta^{-\frac12}u_\beta^{\frac12}).
\]
Applying \(w\) sends this to the corresponding expression with
\(e^{w\beta/2}\). Since \(t_\beta,u_\beta\) are constant on \(W\)-orbits, this
is precisely the divisor \(\widetilde T_{w\beta}\).

Now fix \(u=wy\), and suppose $\alpha\in D(u)$. Equivalently, \(y^{-1}w^{-1}\alpha\in S^-\). We must show that
\(P_w\,{}^wQ_y\) vanishes on \(\widetilde T_\alpha\). There are two cases.

If \(w^{-1}\alpha\in S^+\), then $w^{-1}\alpha\in D(y)$
so \(Q_y\) vanishes on \(\widetilde T_{w^{-1}\alpha}\). By the
\(W\)-equivariance of the divisors, \({}^wQ_y\) vanishes on
\(\widetilde T_\alpha\), and hence so does \(P_w\,{}^wQ_y\).

If \(w^{-1}\alpha\in S^-\), then $\alpha\in D(w)$,
so \(P_w\) vanishes on \(\widetilde T_\alpha\). Hence
\(P_w\,{}^wQ_y\) similarly vanishes on \(\widetilde T_\alpha\).

Thus each summand in \(F_u\) vanishes on the required divisor, so \(F_u\)
satisfies condition (3). Since conditions (1) and (2) are already preserved
under multiplication by \cite[Theorem 1.4]{GKV}, we conclude that
\(\mathbf{H}_{q,\underline{\kappa}}\) is an algebra.

\textbf{Step 2}. We show that \(\mathbf{H}_{q,\underline{\kappa}}\) has a left
\(\mathbf F[T]\)-basis given by the elements \(\beta_w\), \(w\in W\). Recall that
\[
\beta_i=c_i(e^{\alpha_i})[s_i]+d_i(e^{\alpha_i})[1],
\qquad 0\leq i\leq n.
\]
Since the \(\beta_i\)'s are the images of the generators \(T_i\) under Noumi's
representation, they satisfy the affine \(\widetilde C_n\) braid relations.
Thus, for a reduced expression \(w=s_{i_1}\cdots s_{i_r}\), the element
\begin{equation}\label{beta_w basis}
    \beta_w:=\beta_{i_1}\cdots\beta_{i_r}
\end{equation}
is well-defined. Next, observe \(\beta_i\in \mathbf{H}_{q,\underline{\kappa}}\). Indeed, we see
\begin{equation*}
    \text{Res}_{T_{\alpha_i}}(c_i(e^{\alpha_i})+d_i(e^{\alpha_i}))=\text{Res}_{T_{\alpha_i}}(t_i^{1/2})=0
\end{equation*}
since the scalar $t_i^{1/2}$ is regular. Moreover, \(c_i(e^{\alpha_i})\) has exactly the zeros required by condition (3), and both coefficients have no poles except possibly along \(T_{\alpha_i}\). Hence, by Step 1, $\beta_w\in \mathbf{H}_{q,\underline{\kappa}}$ for all \(w\in W\).

For \(\alpha\in S_{\mathrm{Real}}^+\), define
\[
\theta_\alpha:=
\begin{cases}
t^{-\frac12}\dfrac{1-te^\alpha}{1-e^\alpha},
& \alpha/2\notin \tilde{S},\\[10pt]
t_\alpha^{-\frac12}
\dfrac{
(1-u_\alpha^{\frac12}t_\alpha^{\frac12}e^{\alpha/2})
(1+u_\alpha^{-\frac12}t_\alpha^{\frac12}e^{\alpha/2})
}
{1-e^\alpha},
& \alpha/2\in \tilde{S}.
\end{cases}
\]
For \(w\in W\), set
\[
\theta_w:=\prod_{\alpha\in D(w)}\theta_\alpha.
\]
Define
$$(\mathbf{F}[T]_{\mathrm{loc}}\rtimes W)_{\leq w}:=\{\sum_{v\in W} f_v[v]\in\mathbf{F}[T]_{\mathrm{loc}}\rtimes W: f_v=0\text{ unless }v\leq w\}$$
\noindent and similarly define $(\mathbf{F}[T]_{\mathrm{loc}}\rtimes W)_{< w}$. Then we have (compare with \cite[Lemma 2.8]{GKV})
\[
\beta_w-\theta_w[w]\in(\mathbf{F}[T]_{\mathrm{loc}}\rtimes W)_{< w}.
\]
Indeed, if \(w=s_{i_1}\cdots s_{i_r}\) is reduced, the coefficient of \([w]\)
in \(\beta_w\) is
\[
c_{i_1}(e^{\alpha_{i_1}})
s_{i_1}\!\left(c_{i_2}(e^{\alpha_{i_2}})\right)
\cdots
s_{i_1}\cdots s_{i_{r-1}}\!\left(c_{i_r}(e^{\alpha_{i_r}})\right),
\]
and the corresponding roots are exactly the elements of \(D(w)\).

Now let
\[
f=\sum_{y\leq w}f_y[y]\in \mathbf{H}_{q,\underline{\kappa}}\cap(\mathbf{F}[T]_{\mathrm{loc}}\rtimes W)_{\leq w}.
\]
 We claim that $f_w\in \mathbf F[T]\theta_w.$
If \(\alpha\notin D(w)\), then \(s_\alpha w>w\), so \(f_{s_\alpha w}=0\). By the
residue condition, $\operatorname{Res}_{T_\alpha}(f_w)=0,$
and hence \(f_w\) has no pole along \(T_\alpha\). Thus the only possible poles of
\(f_w\) occur along \(T_\alpha\) with \(\alpha\in D(w)\). For such \(\alpha\), condition (3) says that \(f_w\) vanishes along
\(\widetilde T_\alpha\), which is exactly the zero divisor of \(\theta_\alpha\).
Thus the possible poles introduced by dividing by the numerator of
\(\theta_\alpha\) are canceled by the vanishing of \(f_w\). On the other hand,
the denominator \(1-e^\alpha\) of \(\theta_\alpha\) cancels the allowed simple
pole of \(f_w\) along \(T_\alpha\). Dividing by all \(\theta_\alpha\),
\(\alpha\in D(w)\), therefore leaves no poles, and hence
\[
f_w/\theta_w\in \mathbf F[T].
\]

Write \(f_w=a_w\theta_w\) with \(a_w\in\mathbf F[T]\). Using the lower-triangular
expansion of \(\beta_w\), we get
\[
f-a_w\beta_w\in \mathbf{H}_{q,\underline{\kappa},<w}.
\]
Induction on Bruhat order shows that every element of
\(\mathbf{H}_{q,\underline{\kappa}}\) lies in the \(\mathbf F[T]\)-span of the
\(\beta_w\)'s.

Finally, the same triangular expansion gives linear independence. If
\[
\sum_w a_w\beta_w=0
\]
is a finite relation, choose \(w\) maximal with \(a_w\neq0\). The \([w]\)-coefficient
is \(a_w\theta_w\), so \(a_w=0\), a contradiction. Therefore the \(\beta_w\)'s form
a left \(\mathbf F[T]\)-basis of \(\mathbf{H}_{q,\underline{\kappa}}\).

\textbf{Step 3}. It remains to identify \(\mathbf{H}_{q,\underline{\kappa}}\)
with \(H_{q,\underline{\kappa}}(CC_n^\vee)\). By Sahi's faithfulness theorem
\cite[Theorem 3.2]{Sahi1999}, we may identify
\(H_{q,\underline{\kappa}}(CC_n^\vee)\) with its image under the Noumi
polynomial representation. Equivalently, we regard
\(H_{q,\underline{\kappa}}(CC_n^\vee)\) as a subalgebra of
\(\mathbf{F}[T]_{\mathrm{loc}}\rtimes W\), where
\[
T_i\mapsto \beta_i
=
c_i(e^{\alpha_i})[s_i]+d_i(e^{\alpha_i})[1],
\qquad
X_j\mapsto X_j.
\]

\noindent The PBW theorem for the DAHA thus gives
\[
H_{q,\underline{\kappa}}(CC_n^\vee)
=
\bigoplus_{w\in W}\mathbf F[T]\beta_w
\]
as a left \(\mathbf F[T]\)-module. On the other hand, we showed in Step 2, $\mathbf{H}_{q,\underline{\kappa}}$ admits the same $\mathbf{F}[T]$-basis of $\beta_w,w\in W$. Thus the two algebras are the same subspace of \(\mathbf{F}[T]_{\mathrm{loc}}\rtimes W\), and hence
\[
\mathbf{H}_{q,\underline{\kappa}}
\simeq
H_{q,\underline{\kappa}}(CC_n^\vee).\qedhere
\]
\end{proof}

We end this section by commenting that we could have indexed the zeros and residues condition using the other underlying reduced root-subsystem, $S^\vee$, and consequently proven Theorem \ref{zeros residues GKV} using the Noumi generators $\beta_i^\vee$. These approaches are equivalent, and one may see this by noting $\alpha_0^\vee=\alpha_0/2$ and the dual Noumi generators have the same zeros and residues because
$$T_0^\vee=T_0^{-1}e^{-\alpha_0^\vee},\quad T_i^\vee=T_i,\quad T_n^\vee=e^{-\alpha_n^\vee}T_n^{-1}.$$

\subsection{Zeros and Residues via \cite{BaranovskyEvensGinzburgQuantumToriDAHA}}
Finally, let us explain how to convert the zeros-and-residues construction of Theorem~\ref{zeros residues GKV} into the one analogous to \cite{BaranovskyEvensGinzburgQuantumToriDAHA}. Thus, we replace the affine Weyl group $W$ by its finite Weyl group $W_0$ at the cost of replacing the commutative algebra $\mathbf{F}[T]_{\mathrm{loc}}$ by a non-commutative algebra of $q$-difference operators on $T$.

\begin{definition}
For \(\lambda=\sum_{i=1}^n m_i\epsilon_i\in\Lambda\), define the $q$-difference operator
\[
D_{q,\lambda}:=D_{q,\epsilon_1}^{m_1}\cdots D_{q,\epsilon_n}^{m_n},
\]
where
\[
D_{q,\epsilon_i}f(X_1,\dots,X_n)
=
f(X_1,\dots,qX_i,\dots,X_n).
\]
\end{definition}

Thus \(D_{q,\lambda}\) is identified with the affine translation \(\tau(\lambda)\) as defined in \ref{affine translation}, and 
\[
D_{q,\epsilon_i}X_j=q^{\delta_{ij}}X_jD_{q,\epsilon_i}.
\]
Let
\[
\mathcal D_q:=
\bigoplus_{\mu\in\Lambda}\mathbf{F}[T]_{\mathrm{loc}}D_{q,\mu}
\]
be the algebra of rational \(q\)-difference operators on $T$. We then have
\[
\mathcal D_q\rtimes W_0
=
\bigoplus_{w\in W_0,\ \mu\in\Lambda}
\mathbf{F}[T]_{\mathrm{loc}}D_{q,\mu}[w]\simeq\mathbf{F}[T]_{\mathrm{loc}}\rtimes W,
\]
and every element of $\mathcal{D}_q\rtimes W_0$ has a unique expansion \(D=\sum_{w\in W_0,\ \mu\in\Lambda}h_{w,\mu}D_{q,\mu}[w]\) for some  \(h_{w,\mu}\in\mathbf{F}[T]_{\mathrm{loc}}.\) Now let
\[
\alpha=\bar\alpha+k\delta\in S_{\mathrm{Real}}^+
\]
be an affine real root, where \(\bar\alpha\in R\) denotes its finite part. Since
\(e^\delta=q\), we have
\begin{equation}\label{divisor BEG}
    T_\alpha=\{e^\alpha=1\}
=
T_{\bar\alpha,q^{-k}}.
\end{equation}
Moreover, with our convention for affine translations, \(s_\alpha
=
\tau\!\left(-k\bar\alpha^\vee\right)s_{\bar\alpha}\).
Equivalently,
\begin{equation}\label{residues BEG}
s_\alpha D_{q,\mu}[w]
=
D_{q,\;s_{\bar\alpha}\mu-k\bar\alpha^\vee}
[s_{\bar\alpha}w].    
\end{equation}
Here is the main definition and theorem:

\begin{definition}\label{zeros and residue of DAHA BEG}
Consider the \(\mathbf F\)-linear subspace \(\mathcal H_{q,\underline{\kappa}}\) of $\mathcal D_q\rtimes W_0$
consisting of all elements $D=\sum_{w\in W_0,\ \mu\in\Lambda}h_{w,\mu}D_{q,\mu}[w]$ satisfying the following three conditions:

\noindent\textbf{1.}
Each coefficient \(h_{w,\mu}\) has at most first-order poles
along the divisors
$$
T_{\bar\alpha,q^{-k}},\;\text{
where } \alpha=\bar\alpha+k\delta\in S_{\mathrm{Real}}^+.
$$

\noindent\textbf{2.}
For every affine real root $\alpha=\bar\alpha+k\delta\in S_{\mathrm{Real}}^+,$
we have
\[
\operatorname{Res}_{T_{\bar\alpha,q^{-k}}}(h_{w,\mu})
+
\operatorname{Res}_{T_{\bar\alpha,q^{-k}}}
\left(
h_{s_{\bar\alpha}w,\;s_{\bar\alpha}\mu-k\bar\alpha^\vee}
\right)
=0.
\]

\medskip

\noindent\textbf{3.}
For \(\alpha\in R(C_n)^+\), \(w\in W_0\), and \(\mu\in\Lambda\), set
$m_\alpha:=\langle \alpha,\mu\rangle,\;\;
\epsilon_\alpha(w):=
\begin{cases}
1,& w^{-1}\alpha\in R(C_n)^-,\\
0,& w^{-1}\alpha\in R(C_n)^+.
\end{cases}
$

\smallskip

\noindent If \(\alpha=\epsilon_i\pm\epsilon_j\), \(i<j\) is a short root, then
\(h_{w,\mu}\) vanishes on \(T_{\alpha,p}\) for the following values of \(p\):
\[
\begin{aligned}
p&=q^{-r}t^{-1},
&\qquad m_\alpha&<0,
&\qquad 0\leq r&\leq -m_\alpha-1+\epsilon_\alpha(w),\\
p&=t^{-1},
&\qquad m_\alpha&=0,
&\qquad \epsilon_\alpha(w)&=1,\\
p&=q^r t,
&\qquad m_\alpha&>0,
&\qquad 1\leq r&\leq m_\alpha-\epsilon_\alpha(w).
\end{aligned}
\]

\noindent If \(\alpha=2\epsilon_i\) is a long root, then \(h_{w,\mu}\) vanishes on \(T_{\frac{\alpha}{2},p}:=\{e^{\frac{\alpha}{2}}=p\}\) for the following values of \(p\):
\[
\begin{aligned}
p&\in
\left\{
q^{-r}t_n^{-\frac12}u_n^{-\frac12},
-q^{-r}t_n^{-\frac12}u_n^{\frac12}
\right\},
& m_\alpha&<0,
& 0\leq r&\leq -\frac{m_\alpha}{2}-1+\epsilon_\alpha(w),\\
p&\in
\left\{
q^{-(r+\frac12)}t_0^{-\frac12}u_0^{-\frac12},
-q^{-(r+\frac12)}t_0^{-\frac12}u_0^{\frac12}
\right\},
& m_\alpha&<0,
& 0\leq r&\leq -\frac{m_\alpha}{2}-1,\\
p&\in
\left\{
t_n^{-\frac12}u_n^{-\frac12},
-t_n^{-\frac12}u_n^{\frac12}
\right\},
& m_\alpha&=0,
& \epsilon_\alpha(w)&=1,\\
p&\in
\left\{
q^r t_n^{\frac12}u_n^{\frac12},
-q^r t_n^{\frac12}u_n^{-\frac12}
\right\},
& m_\alpha&>0,
& 1\leq r&\leq \frac{m_\alpha}{2}-\epsilon_\alpha(w),\\
p&\in
\left\{
q^{r+\frac12}t_0^{\frac12}u_0^{\frac12},
-q^{r+\frac12}t_0^{\frac12}u_0^{-\frac12}
\right\},
& m_\alpha&>0,
& 0\leq r&\leq \frac{m_\alpha}{2}-1.
\end{aligned}
\]

\end{definition}

\begin{thm}\label{BEG zeros residues theorem}
    The subspace $\mathcal{H}_{q,\underline{\kappa}}$ of Definition \ref{zeros and residue of DAHA BEG} is an algebra isomorphic to $H_{q,\underline{\kappa}}(CC_n^\vee)$.
\end{thm}
\begin{proof}
We just need to show that conditions \(1\), \(2\), and \(3\) of Definition~\ref{zeros and residue of DAHA BEG} are equivalent to conditions
\(1\), \(2\), and \(3\), respectively, of
Definition~\ref{zeros and residue of DAHA}.

Equation~\eqref{divisor BEG} identifies the divisors appearing in condition
\(1\) of the two definitions, while Equation~\eqref{residues BEG} identifies
the pairs of coefficients appearing in condition \(2\). Thus conditions
\(1\) and \(2\) of Definition~\ref{zeros and residue of DAHA BEG} are
equivalent to conditions \(1\) and \(2\), respectively, of
Definition~\ref{zeros and residue of DAHA}. It remains to prove the
equivalence of condition \(3\).

In the short root case, namely $\alpha=\ep_i\pm \ep_j$, our condition is precisely that of \cite{BaranovskyEvensGinzburgQuantumToriDAHA}\footnote{with the power of $q$ renormalized}, hence that case is already done.

We now check condition \textup{(3)} of GKV in the case $\alpha=2\ep_i$. For an affine root $\gamma=\alpha+k\delta$, we have
\[
(\tau(\mu)w)^{-1}\gamma
=
w^{-1}\alpha+\bigl(k+\langle \alpha,\mu\rangle\bigr)\delta.
\]
We have 2 cases: $k=2r$ is even or $k=2r+1$ is odd. Suppose further that $r\geq 0$ in either case; the other cases are handled similarly. 

For the first case, $(\tau(\mu)w)^{-1}\gamma=w^{-1}\alpha+(2r+m_\alpha)\delta,$ so $(\tau(\mu)w)^{-1}\gamma<0$ precisely when $m_\alpha<0$ and $0\leq r\leq -\frac{m_\alpha}{2}-1+\epsilon_\alpha(w).$
Since \(e^{\gamma/2}=q^r e^{\alpha/2}\), the corresponding GKV zero divisors are given by
$$
q^r e^{\alpha/2}=t_n^{-\frac12}u_n^{-\frac12}
\qquad\text{or}\qquad
q^r e^{\alpha/2}=-t_n^{-\frac12}u_n^{\frac12}.
$$
This is the zero locus in the first line of BEG condition 3.

For the second case, $(\tau(\mu)w)^{-1}\gamma=w^{-1}\alpha+(2r+1+m_\alpha)\delta.$ Since $m_\alpha$ is even, this is negative
precisely when $0\leq r\leq -\frac{m_\alpha}{2}-1.$ Since $e^{\gamma/2}=q^{r+\frac12}e^{\alpha/2}$, the GKV zero divisors are
$$
q^{r+\frac12}e^{\alpha/2}=t_0^{-\frac12}u_0^{-\frac12}
\qquad\text{or}\qquad
q^{r+\frac12}e^{\alpha/2}=-t_0^{-\frac12}u_0^{\frac12}.
$$
This is the zero locus in the second line of BEG condition 3. The cases of $m_\alpha=0,m_\alpha>0$ are treated similarly and this completes the proof.
\end{proof}
\noindent Finally, we note that one may obtain a zeros and residues construction of the spherical subalgebra by further enforcing similar constraints as in \cite[Theorem 8.1]{BaranovskyEvensGinzburgQuantumToriDAHA}.
\section{Deformation theory of the $CC_n^\vee$-DAHA}

In this section, we establish that the DAHA
$H_{q,\underline{\kappa}}(CC_n^\vee)$ is the formal universal deformation of a quantum torus algebra crossed product with a finite Weyl group. For $n=1$, the same argument, with the parameter $t$ omitted, gives a new proof of Oblomkov's universality theorem \cite{Oblomkov2004}, using the Hochschild cohomology computation of Etingof--Oblomkov \cite{EO}. For convenience, however, throughout this section we assume $n>1$, so that the deformation has six parameters. We first recall some well-known definitions and results in deformation theory; see, for example, \cite{Schedler2012} for an exposition. Throughout, $A$ will denote an associative $\Cb$-algebra.
\begin{definition}
    An infinitesimal deformation of $A$ is a $\Cb[\ep]/\ep^2$-module $A[\ep]/\ep^2:=A\otimes_{\mathbf{C}}\mathbf{C}[\ep]/\ep^2$, equipped with a product, $\star$, making $A[\ep]/\ep^2$ an associative algebra and such that 
    $a\star b=ab\mod\ep$. We denote by $(A[\ep]/\ep^2,\star)$ the infinitesimal deformation of $A$.
\end{definition}
\begin{definition}
    A 1-parameter formal deformation of $A$ is an associative algebra $A_{\hbar}:=A\hat\otimes_{\mathbf{C}}\mathbf{C}[[\hbar]]$ equipped with a $\Cb[[\hbar]]$-bilinear product $\star$ such that $a\star b=ab\mod\hbar$.
\end{definition}
More generally, suppose $(R,\mathfrak{m})$ is a complete augmented $\mathbf{C}$-algebra with augmentation ideal $\mathfrak{m}$. Denote the completed tensor product
$$A\hat{\otimes}R:=\lim_{n\rightarrow\infty}A\otimes R/\mathfrak{m}^n.$$

\begin{definition}
    A formal deformation of $A$ over a complete augmented $\Cb$-algebra $R$ with maximal ideal $\mathfrak{m}\subset R$ is an $R$-algebra $A'$ isomorphic to $A\hat{\otimes}R$ as $R$-modules such that 
    $$A'\otimes_R (R/\mathfrak{m})\simeq A\;\text{ as $\Cb$-algebras}.$$
\end{definition}

Given a continuous homomorphism $p:R\rightarrow R'$ and a formal deformation $(A\hat{\otimes}R,\star)$, we may base change using $p$ to obtain a new formal deformation $(A\hat{\otimes}R',p(\star))$ with product
$$ap(\star)b:=(\text{Id}\otimes p)(a\star b)$$
\begin{definition}
    A universal formal deformation $(A\hat{\otimes}R,\star)$ is a formal deformation such that for every other formal deformation $(A\hat{\otimes}R',\star')$, there exists a unique continuous homomorphism $p:R\rightarrow R'$ such that $(A\hat{\otimes}R',\star')$ is equivalent to the base-change deformation $(A\hat{\otimes}R',p(\star))$. If we only require existence, and not uniqueness, of a homomorphism $p:R\rightarrow R'$, then we say the deformation is \textit{versal}.
\end{definition}

Let $HH^\bullet(A):=HH^\bullet(A,A)$ denote the Hochschild cohomology of $A$. Given a vector space $V$, let $\hat{\mathcal{O}}[V]$ denote the completed polynomial algebra in $\dim(V)$ many variables. The following theorem is well-known:

\begin{thm}\label{H^1 and H^3=0}If $HH^3(A,A)=0$, then there exists a versal deformation of $A$ over $R:=\hat{\mathcal{O}}[HH^2(A)]$. If $HH^1(A)=0$, this deformation is universal
\end{thm}
\noindent We recall that the 2nd degree Hochschild cohomology controls infinitesimal deformations, and the third degree controls obstructions to lifting the deformation. Furthermore, the 1st degree Hochschild cohomology controls infinitesimal automorphisms of the deformation.

Following the notation introduced in Section 2, we remind the affine root system $\tilde{S}$ is of type $CC_n^\vee$, $S$ of type $C_n$, and $W_0<W:=W_{\text{aff}}(C_n)$ is the corresponding finite, respectively, affine Weyl group. Let ${H}_{q,\underline{\kappa}}(CC_n^\vee)$ denote the corresponding DAHA. 

Now, define the complete ground ring $R$ with maximal ideal $\mathfrak{m}$ by 
\begin{equation}\label{ground ring R}
    R:=\Cb[[\hbar,t_0',u_0',t_n',u_n',t']],\; \mathfrak{m}=(\hbar,t_0',u_0',t_n',u_n',t').
\end{equation}

Recall the ground ring $\mathbf{F}$ in Equation \ref{ground ring}. Consider the completed DAHA 
$$\hat{H}_{q,\underline{\kappa}}(CC_n^\vee):=H_{q,\underline{\kappa}}(CC_n^\vee)\hat\otimes_{\mathbf{F}_{\Cb}} R$$
whose parameters are 
\begin{equation}\label{formal parameters}
(q,\underline{\kappa}):=(q,t_0,u_0,t_n,u_n,t)=(q_0e^{\hbar},e^{t_0'},e^{u_0'},e^{t_n'}, e^{u_n'},e^{t'}).
\end{equation}
Here, the exponentials $e^u\in\Cb[[u]]$ are viewed formally, so each parameter lives in $R$. Also, the map $\mathbf{F}\rightarrow R$ is determined by $q^{\frac{1}{2}}\mapsto q_0^{\frac{1}{2}}e^{\frac{\hbar}{2}}$ and similarly for the remaining parameters.

\begin{definition}Define the quantum torus as the algebra
\begin{align*}B_{{q}}[n]:= \mathbf{C}\langle X_1^{\pm1},Y_1^{\pm1},\dots, X_n^{\pm1},Y_n^{\pm1}\rangle/ \big(&Y_iX_j=q^{\delta_{ij}}X_jY_i,\; X_i X_j=X_j X_i, Y_i Y_j = Y_j Y_i)
\end{align*}
\end{definition}
\noindent When the context is clear, we will omit $n$ from the notation and write $B_q=B_q[n]$. 

Denote the algebra $A$ the result of specializing $(q,\underline{\kappa})=(q_0,\underline{1})$. Thus we find 
\begin{equation}\label{undeformed algebra}
    A:= \hat{H}_{q,\underline{\kappa}}(CC_n^\vee)/(\mathfrak{m})\simeq {H}_{q_0,\underline{1}}(CC_n^\vee)=\mathbf{F}[T]\rtimes W \simeq B_{q_0}[n]\rtimes W_0.
\end{equation}
\noindent Finally, define the completion
$$\hat{A}:=A\hat{\otimes}R.$$

Let us now discuss the Hochschild cohomology of $A$. For $w\in W_0$, let $B_q.w:=\{aw:a\in B_q\}$ be the $B_q\otimes B_q$-module defined by 
$$(a\otimes b).(cw):= acwb = ac({}^wb)w$$
for $a,b,c\in B_q$. Next, we use a standard decomposition (see e.g. \cite{SheplerWitherspoonLie},\cite{Stefan1995}) for the Hochschild cohomology for skew group algebras to write
\begin{align}\label{isotypic decomposition HH}
    HH^\bullet(B_{q_0}[n]\rtimes W_0)&\simeq HH^\bullet(B_{q_0}[n],B_{q_0}[n]\rtimes W_0)^{W_0}\\
    &\simeq    \bigg( \bigoplus_{w\in  W_0}HH^\bullet(B_{q_0}[n],B_{q_0}[n].w)\bigg)^{W_0}\\
    &\simeq \bigoplus_{[w]\in \mathcal{C}}HH^\bullet(B_{q_0}[n],B_{q_0}[n].w)^{Z(w)}
\end{align}
\noindent where $Z(w)\subset W_0$ is the centralizer of $w$ and $\mathcal{C}$ denotes the set of conjugacy classes in $W_0$.

\noindent We denote by $\text{pr}_{[w]}$ the projection to the $[w]$-conjugacy class
\begin{equation}\label{conjugacy class projection}
    \text{pr}_{[w]}:HH^\bullet(B_{q_0}[n]\rtimes W(C_n))\rightarrow HH^\bullet(B_{q_0}[n],B_{q_0}[n]w)^{Z(w)}
\end{equation}
\noindent Similarly, given $a=\sum_{w\in W_0}a_w[w]\in A=\oplus_{w\in W_0}B_qw$, denote the corresponding projection map by
\begin{equation}\label{pr_w}
\text{pr}_w:A\rightarrow B_qw,\;\; a\mapsto a_w.
\end{equation}

Thus, to compute $HH^\bullet(A)$, we are reduced to computing equivariant cohomology of the quantum torus algebra $B_q$ with coefficients valued in the $B_q$-bimodule $B_q.w$ for $w\in W_0$. In low rank, this is easily computable using the Koszul complex which we now recall.

\subsection{Quantum Koszul resolution}\label{Koszul resolution}
Denote \(Z_1,\dots,Z_{2n}\) to be the generators
\(X_1,\dots,X_n,Y_1,\dots,Y_n\) of \(B_q\). Define the vector space
\[
V:=\operatorname{span}\{e_1,\dots,e_{2n}\},
\]
where the \(e_i\)'s are a basis corresponding to generators
\(Z_i\). Let $\bigwedge_q V$ denote the quantum exterior algebra. Recall since $W_0$ is finite, 
$$HH^\bullet(A,A) = HH^\bullet(B_q, A)^{W_0},$$
so we now consider a Koszul resolution of $B_q$. Define
\[
K_d(B_q):=B_q\otimes\bigwedge_q^dV\otimes B_q.
\]
Its differential $d_m:K_m(B_q)\rightarrow K_{m-1}(B_q)$ is given by
\begin{align}
&d_m\left(1\otimes e_{j_1}\wedge\cdots\wedge e_{j_m}\otimes1\right)\notag=
\sum_{r=1}^m(-1)^{r+1}
\left[
\left(\prod_{s=1}^r q_{j_sj_r}\right)
Z_{j_r}\otimes
e_{j_1}\wedge\cdots\wedge\widehat{e_{j_r}}\wedge\cdots\wedge e_{j_m}
\otimes1\right.\notag\\
&\hspace{55mm}\left.
-
\left(\prod_{s=r}^m q_{j_rj_s}\right)
1\otimes
e_{j_1}\wedge\cdots\wedge\widehat{e_{j_r}}\wedge\cdots\wedge e_{j_m}
\otimes Z_{j_r}
\right],
\label{Koszul differential}
\end{align}
where \(j_1<\cdots<j_m\) and the scalars \(q_{ij}\) are determined by the quantum torus relations $B_q$: $Z_iZ_j=q_{ij}Z_jZ_i.$ Thus, for any \(B_q\)-bimodule \(M\), $$HH^i(B_q,M)
\simeq
H^i\left(\operatorname{Hom}_{B_q^e}(K_\bullet(B_q),M)\right).$$

We may in fact directly compare the Koszul complex with the Hochschild complex by choosing a comparison map (see \cite[Section 4]{WitherspoonZhouGerstenhaber})
\[
\Psi_\bullet:\operatorname{Bar}_\bullet(B_q)\longrightarrow K_\bullet(B_q)
\]
lifting the identity on \(B_q\). We may choose \(\Psi_1\) and \(\Psi_2\) satisfying
\begin{align*}
\Psi_1(1\otimes Z_i\otimes1)
&=
1\otimes e_i\otimes1,
\qquad
\Psi_1(1\otimes Z_i^{-1}\otimes1)
=
-Z_i^{-1}\otimes e_i\otimes Z_i^{-1}.\\
d_K\Psi_2&=\Psi_1d_{\mathrm{Bar}}.
\end{align*}

A Koszul $2$-cocycle with values in $M$ may be identified with a linear map
\[
\gamma:\wedge_q^2V\longrightarrow M,
\]
such that $\gamma\circ d_3=0$. This determines a Hochschild \(2\)-cocycle via
\[
\widetilde\gamma:=\gamma\circ\Psi_2.
\]

The Weyl group \(W_0\) does not in general act linearly on \(V\), since an
element of \(W_0\) may send a generator \(Z_i\) to \(Z_j^{-1}\), whereas
\(V\) contains only the formal symbols \(e_1,\dots,e_{2n}\). Nevertheless,
the comparison map transports the \(W_0\)-action on the bar complex to the
Koszul complex. We call a Koszul \(2\)-cocycle, \(\gamma\), \(W_0\)-equivariant if the
corresponding Hochschild cocycle
\[
\widetilde\gamma:=\gamma\circ\Psi_2
\]
is $W_0$-equivariant. Namely, for every $w\in W_0$ and  $a,b\in B_{q},$
\[
\gamma\!\left(\Psi_2(1\otimes w(a)\otimes w(b)\otimes1)\right)
=
w\,\gamma\!\left(\Psi_2(1\otimes a\otimes b\otimes1)\right)w^{-1}.
\]
For example, if \(w(X_i)=X_i^{-1}\) and \(w(X_j)=X_j\), and $i<j\leq n$, then $W_0$-equivariance on $\gamma$ becomes
\[
w\,\gamma(e_i\wedge e_j)\,w^{-1}
=
-X_i^{-1}\gamma(e_i\wedge e_j)X_i^{-1}.
\]

\subsection{Proof of main theorem}
The entirety of this section is devoted to proving the following main theorem.
\begin{thm}\label{DAHA universal deformation}
Suppose $q=q_0e^\hbar$ and $q_0$ is not a root of unity. Then the DAHA $\hat{H}_{q,\underline{\mathbf{\kappa}}}(CC_n^\vee)$ is the universal formal deformation of $A=B_{q_0}[n]\rtimes W_0$ over $R$ as in \ref{ground ring R}.
\end{thm}
\begin{proof}
Recall the $\mathbf{F}[T]$-module basis $\beta_w$, $w\in W$, introduced in Equation \ref{beta_w basis}. By writing the parameters formally using \ref{formal parameters}, we may view $\beta_w$ as elements of the algebra $R\otimes_{\mathbf{F}_{\Cb}}(\mathbf{F}[T]_{\mathrm{loc}}\rtimes W).$ Now, define the map of $R$-modules
$$\sigma_{\beta_{\underline{\kappa}}}:\hat{A}\rightarrow\hat{H}_{q,\underline{\kappa}}(CC_n^\vee),\quad \sum_{w\in W} f_w[w]\mapsto\sum_{w\in W}f_w\beta_w.$$
Note, here we are identifying $H_{q,\underline{\kappa}}(CC_n^\vee)$ with its image in $R\otimes_{\mathbf{F}_{\Cb}}(\mathbf{F}[T]_{\mathrm{loc}}\rtimes W)$ using the Noumi representation. By proof of Theorem \ref{zeros residues GKV}, the $\beta_w$ are an $R\otimes_{\mathbf{F}}\mathbf{F}[T]$-module basis for $\hat{H}_{q,\underline{\kappa}}(CC_n^\vee)$. The elements $[w]$ are also an $R\otimes_{\mathbf{F}}\mathbf{F}[T]$-module basis for $R\otimes A$. Thus, $\sigma_{\beta_{\underline{\kappa}}}$ is an $R$-module isomorphism. Moreover, since $\beta_w=[w]\mod\mathfrak{m}$, $\sigma_\beta$ is identity mod $\mathfrak{m}$. We conclude $\hat{H}_{q,\underline{\kappa}}(CC_n^\vee)$ is a formal deformation of $A$ over $R$. Moreover, it is a flat deformation because it has a PBW basis.

Now, define a new product, $*_{\beta_{\underline{\kappa}}}$, on $A\hat{\otimes} R$ by transporting the product of $\hat{H}_{q,\underline{\kappa}}(CC_n^\vee).$ Namely,
$$a*_{\beta_{\underline{\kappa}}} b:=\sigma_{\beta_{\underline{\kappa}}}^{-1}(\sigma_{\beta_{\underline{\kappa}}}(a)\sigma_{\beta_{\underline{\kappa}}}(b)),\quad \text{ for }a,b\in \hat{A}$$
With this product, $\sigma_{\beta_{\underline{\kappa}}}$ is by definition an isomorphism of algebras
$$\sigma_{\beta_{\underline{\kappa}}}:(\hat{A},*_{\beta_{\underline{\kappa}}})\rightarrow\hat{H}_{q,\underline{\kappa}}(CC_n^\vee).$$
Since $\sigma_{\beta_{\underline{\kappa}}}$ is identity mod $\mathfrak{m}$, we may write the product mod $\mathfrak{m}^2$ as 
$$a*_{\beta_{\underline{\kappa}}} b=ab+\mu_{\beta_{\underline{\kappa}}}(a,b)\mod\mathfrak{m}^2,\;\text{ for }a,b\in A$$
where $\mu_{\beta_{\underline{\kappa}}}$ is some linear map
$$\mu_{\beta_{\underline{\kappa}}}:A\otimes A\rightarrow A\otimes\mathfrak{m}/\mathfrak{m}^2.$$
Since $\hat{H}_{q,\underline{\kappa}}(CC_n^\vee)$ is an associative algebra, $\mu_{\beta_{\underline{\kappa}}}$ is a Hochschild 2-cocycle and defines a class \cite{WitherspoonBook} $[\mu_{\beta_{\underline{\kappa}}}]\in HH^2(A)\otimes\mathfrak{m}/\mathfrak{m}^2.$ Equivalently, $[\mu_{\beta_{\underline{\kappa}}}]$ induces a map \begin{equation}\label{Kodaira Spencer}
    KS:(\mathfrak{m}/\mathfrak{m}^2)^*\rightarrow HH^2(A)
\end{equation}
Note $T_0(R)=(\mathfrak{m}/\mathfrak{m}^2)^*$ and this map is the usual Kodaira-Spencer morphism in deformation theory. In Theorem \ref{isom of HH^2}, we will show this morphism is an isomorphism. Thus, spanning over all parameters $(q,\underline{\kappa})\in\Cb^6$, we deduce $\hat{H}_{q,\underline{\kappa}}(CC_n^\vee)/\mathfrak{m}^2$ produces all infinitesimal deformations of $A=\hat{A}/( \mathfrak{m})$. 

Next, we may identify
\begin{equation}\label{t=1 limit rewritten}
   A= B_{q_0}[n]\rtimes W_0\simeq (B_{q_0}[1]\rtimes \Zb/2\Zb)^{\otimes n}\rtimes S_n.
\end{equation} Thus, \cite[Theorem 5.1]{EO} applied to the algebra in \ref{t=1 limit rewritten} implies
\begin{align}
HH^2(A)&=\begin{cases}\mathbf{C}^5 &\text{ if }n=1\\\mathbf{C}^6 &\text{ if }n>1 		\end{cases}\\
HH^i(A)&=0\;\text{ for odd $i$}.
\end{align}

Since \(HH^3(A)=0\), Theorem~\ref{H^1 and H^3=0} implies that the deformation functor of \(A\) is unobstructed and admits a versal deformation over the completed symmetric algebra on \(HH^2(A)^*\). Since the Kodaira--Spencer morphism
\[
(\mathfrak m/\mathfrak m^2)^*\xrightarrow{\sim}HH^2(A)
\]
is an isomorphism, the deformation \(\hat{H}_{q,\underline{\kappa}}(CC_n^\vee)\) is the versal one. Finally, since \(HH^1(A)=0\), it is universal by Theorem~\ref{H^1 and H^3=0}.
\end{proof}

It remains to prove the Kodaira-Spencer morphism is an isomorphism. To do this, let us work over the dual numbers $\Cb[\ep]/\ep^2$. Namely, view the parameters as   
\begin{equation}\label{epsilon parameter}
(q,\underline{\kappa}):=(q_0+\ep q_0',1+\ep\underline{\kappa}'):=(q_0+\ep q_0',1+\ep t_0',1+\ep u_0', 1+\ep t_n', 1+\ep u_n', 1+\ep t').
\end{equation} 
Then we may write $\beta_w=[w]+\ep b_w\in H_{q_0+\ep q_0',\underline{1}+\ep\underline{\kappa'}}(CC_n^\vee)$. Thus we may write the product $*_\beta$ as 
\begin{equation}\label{hochschild cocycle beta}
    a*_\beta b=ab+\ep\mu_\beta(a,b)\;\;\text{ for }a,b\in A
\end{equation}
We will suppress the dependency of $\kappa$ in $*_{\beta_{\underline{\kappa}}},\mu_{\beta_{\underline{\kappa}}}$, when it is clear from context what parameters we specify to. 
\begin{lemma}\label{mu_beta def new}
    The Hochschild 2-cocycle $\mu_\beta$ constructed in \ref{hochschild cocycle beta} takes the following form for elements $f[w],g[v]\in A$, with $f,g\in\mathbf{F}[T]$
    \begin{equation}\label{modified beta cocycle}
        \mu_\beta(f[w],g[v])=f[w]gb_v+fb_wg[v]-fw(g)b_{wv}
    \end{equation}
    where $b_w$ are defined by $\beta_w=[w]+\ep b_w$ as before.
\end{lemma}
\begin{proof}
Working modulo \(\varepsilon^2\), we find
\[
\sigma_\beta(f[w])\sigma_\beta(g[v])
=f([w]+\varepsilon b_w)g([v]+\varepsilon b_v)
=f[w]g[v]+\varepsilon(f[w]gb_v+fb_wg[v])
\]
and
\[
\sigma_\beta(f[w]g[v])
=\sigma_\beta(f\,w(g)[wv])
=f\,w(g)\beta_{wv}
=f\,w(g)[wv]+\varepsilon f\,w(g)b_{wv}.
\]
Hence
\[
\sigma_\beta(f[w])\sigma_\beta(g[v])-\sigma_\beta(f[w]g[v])
=
\varepsilon\bigl(f[w]gb_v+fb_wg[v]-f\,w(g)b_{wv}\bigr).
\]
The result follows. 
\end{proof}
We note that a priori, the terms $b_w,b_v$ and $b_{wv}$ are in $\mathbf{F}[T]_{\mathrm{loc}}$ and have poles. But the expression in \ref{modified beta cocycle} is in $A$ because $\mu_{\beta}(a,b)\in A$ for all $a,b\in A$.

\begin{thm}\label{isom of HH^2}
    Let $(q,\underline{\kappa})\in\Cb^6$ be as in Equation \ref{epsilon parameter}.  Let $\mu_{\beta_{\underline{\kappa}}}$ be defined using Equation \ref{hochschild cocycle beta}. Then the Kodaira-Spencer morphism
    $$(q_0',\underline{\kappa}')\in\Cb^6\mapsto [\mu_{\beta_{\underline{\kappa}}}]\in HH^2(A)$$
\noindent is an isomorphism.
\end{thm}

We will now prove Theorem \ref{isom of HH^2} through a sequence of three lemmas. The lemmas show the projections of the Hochschild cohomological classes $[\mu_{\beta_{\underline{\kappa}}}]$ to the identity, $P_1:=s_1\dots s_n\dots s_1$, and, when $n>1$, the $s_{1}$, respective conjugacy classes of $W_0$, are surjective maps.

First we compute $\text{pr}_{[1]}([\mu_{\beta_{\underline{\kappa}}}])$.
\begin{lemma}\label{id projection}
Suppose 
$$(q,\underline{\kappa})=(q_0+\ep q_0', 1,1,1,1,1).$$
    Then $$HH^2(B_{q_0}[n],B_{q_0}[n])^{W_0}=\Cb=\text{span}_{\Cb}([\sum_{i=1}^nX_iY_i]).$$
   Furthermore,  the map $(q_0',\underline{1})\mapsto\mathrm{pr}_{[1]}([\mu_{\beta_{q,\underline{1}}}]) $ is an isomorphism since 
    $$\mathrm{pr}_{[1]}([\mu_{\beta_{q,\underline{1}}}])=q_0'[\sum_{i=1}^nX_iY_i]\in HH^2(B_{q_0}[n])^{W_0}$$
\end{lemma}
\begin{proof}
    Let $V=\text{span}\{X_i,Y_i,1\leq i \leq n\}$. Using the quantum Koszul complex, we may compute (see for example \cite{WambstQuantumTorus})
    $$HH^2(B_{q_0}[n])=\wedge^2V.$$
    The group $W_0=S_n\rtimes (\Zb/2\Zb)^n$ acts by signed permutations, thus the $W_0$-invariant subspace is spanned by $\sum_{i=1}^n X_i\wedge Y_i$, and this proves the first claim.

Now we prove the 2nd claim. Since all $\underline{\kappa}$ parameters are equal to 1, the only deformed relations are the quantum torus relations. Thus 
\begin{equation*}
    Y_i*_{\beta_{\underline{\kappa}}} X_j = (q_0+\ep q_0')^{\delta_{ij}}X_jY_i\Rightarrow \mu_{\beta_{\kappa}}(Y_i,X_j)=\begin{cases}
        0&\text{ if }i\neq j\\ q_0'X_iY_i &\text{ if }i=j.
    \end{cases}
\end{equation*}
Since the target is 1-dimensional, it is spanned by the image of $q_0'$ and this completes the proof.
\end{proof}

Second, we compute $\text{pr}_{[P_1]}([\mu_{\beta_{q,\underline{\kappa}}}])$ where
$$P_1:=s_1\dots s_n\dots s_1\in W_0$$ is the automorphism which sends $X_1\mapsto X_1^{-1}, Y_1\mapsto Y_1^{-1}$ and fixes $X_i, Y_i,i>1$. As $B_{q_0}[n]$-bimodules, we have 
$$B_{q_0}[n]P_1=B_{q_0}[1]P_1\otimes B_{q_0}[n-1].$$
Now, by a direct computation using the Koszul resolution (discussed in proof of Lemma \ref{P_1 projection} below), we find 
$$HH^0(B_{q_0}[1],B_{q_0}[1]P_1)=HH^1(B_{q_0}[1],B_{q_0}[1]P_1)=0.$$
Also $HH^0(B_{q_0}[n-1])=Z(B_{q_0}[n-1])=\mathbf{C}$ since the center of the quantum torus $B_{q_0}[n]$ at a non-root of unity $q_0$ is trivial. Thus the Kunneth theorem implies
$$HH^2(B_{q_0}[n],B_{q_0}[n]P_1)=HH^2(B_{q_0}[1],B_{q_0}[1]P_1)\otimes HH^0(B_{q_0}[n-1]) = HH^2(B_{q_0}[1],B_{q_0}[1]P_1).$$
Now, the centralizer $Z(P_1)=\langle P_1\rangle\times W_0(C_{n-1})$, where $\langle P_1\rangle=\Zb/2\Zb=W_0(C_1)$. Thus we find
$$HH^2(B_q[n],B_q[n]P_1)^{Z(P_1)}\simeq HH^2(B_q[1],B_q[1]P_1)^{W_0(C_1)}$$

It is now easy to directly compute the cocycle $\mu_\beta$ in this case. Let us write a convenient presentation for $H_{q,\underline{1}}(CC_1^\vee)$. Define $Y=T_1T_0$. Then
$$H_{q,\underline{1}}(CC_1^\vee)=\Cb\langle X^{\pm1},Y^{\pm1},s\rangle\slash (sXs=X^{-1}, sYs=Y^{-1},s^2=1, YX=qXY)$$

\begin{lemma}\label{P_1 projection}
    Suppose $n=1$ and $$(q,t_0,u_0,t_1,u_1)=(q_0, 1+\ep t_0',1+\ep u_0', 1+\ep t_1', 1+\ep u_1').$$ Then the cocycle $\mu_\beta$ for $H_{q,\underline{1}}(CC_1^\vee)$, as defined in Equation \ref{hochschild cocycle beta}, takes the following value
    \begin{equation}\label{mu_beta rank 1}\mu_\beta(Y,X)=q_0^{\frac{1}{2}}u_0'[s]+t_0'[sX]+q_0u_1'[sY]+q_0t_1'[sXY]
    \end{equation}

\noindent Thus, the following map is an isomorphism
\[
\begin{aligned}
\Cb^4 &\longrightarrow HH^2(B_q[1],B_q[1]P_1)^{W_0(C_1)},\\
(t_0',u_0',t_1',u_1')
&\longmapsto
\operatorname{pr}_{[P_1]}([\mu_\beta])
\end{aligned}
\]
\end{lemma}
\begin{proof}
    We proceed by direct computation. First, observe 
    $$t_i^{\frac{1}{2}}=(1+\ep t_i')^{\frac{1}{2}}=1+\ep\frac{t_i'}{2}\mod\ep^2.$$
    Thus substituting the appropriate parameters into formulas in \ref{beta_i basis}, we find 
    \begin{align}
        \beta_1&=[s]+\ep b_1 &&=[s]+\ep\bigg(\frac{\frac{1}{2}t_1'(X^2+1)+u_1'X}{X^2-1}[s]-\frac{u_1'X+t_1'}{X^2-1}\bigg)\\
        \beta_0&=[s_0]+\ep b_0 &&=[s_0]+\ep\bigg(\frac{\frac{1}{2}t_0'(q_0X^{-2}+1)+q_0^{\frac{1}{2}}u_0'X^{-1}}{q_0X^{-2}-1}[s_0]-\frac{q_0^{\frac{1}{2}}u_0'X^{-1}+t_0'}{q_0X^{-2}-1}\bigg)
    \end{align}

   Since $Y=T_1T_0$, we find 
   $$\beta_Y=(s+\ep b_1)(s_0+\ep b_0)=ss_0+\ep(b_1s_0+sb_0)=:ss_0+\ep b_Y$$
Then,
\begin{equation}\label{b_Y term}
    b_Y:=b_1s_0+sb_0= a(X)Y-\frac{t_1'+u_1'X}{X^2-1}Y^{-1}s-\frac{q_0^{\frac{1}{2}}u_0'X+t_0'}{q_0X^{2}-1}[s]
\end{equation}
   for some $a(X)\in \mathbb{C}(X)$. Since $YX=q_0XY$, we find $\sigma_\beta(YX)=\beta_YX=q_0X\beta_y$. Thus
   $$\mu_\beta(Y,X)=\beta_YX-\sigma_\beta(YX)=b_YX-q_0Xb_Y.$$

   Now, we simply plug in Equation \ref{b_Y term} into the above formula and simplify. Indeed, observe the $a(X)Y$ contribution disappears because
   $$a(X)YX-q_0Xa(X)Y=q_0a(X)XY-q_0Xa(X)Y=0.$$
   The $Y^{-1}s$ contribution is
   \begin{align*}
       -\frac{t_1'+u_1'X}{X^2-1}(Y^{-1}sX-q_0XY^{-1}s)&=-\frac{t_1'+u_1'X}{X^2-1}q_0(X^{-1}-X)Y^{-1}s \\
       &= q_0(u_1'+t_1'X^{-1})Y^{-1}s
   \end{align*}
    The $s$ contribution is 
    \begin{align*}
      -\frac{q_0^{\frac{1}{2}}u_0'X+t_0'}{q_0X^{2}-1}(X^{-1}-q_0X)s&=-\frac{q_0^{\frac{1}{2}}u_0'X+t_0'}{q_0X^{2}-1}\frac{1-q_0X^2}{X}s\\
      &=(q_0^{\frac{1}{2}}u_0'+t_0'X^{-1})s
    \end{align*}
  By combining the two simplified lines, and using $sX=X^{-1}s,sY=Y^{-1}s$,  this completes the verification of the formula for $\mu_\beta(Y,X).$
To conclude that the projection is surjective, we may directly compute $HH^\bullet(B_{q_0}[1],B_{q_0}[1]s)$ using the Koszul resolution (see also \cite[Section 5]{Oblomkov2004}). In rank 1, the Koszul complex is
\[
B_{q_0}s
\xrightarrow{d_1^*}
B_{q_0}s\oplus B_{q_0}s
\xrightarrow{d_2^*}
B_{q_0}s
\longrightarrow 0.
\]
Writing elements as \(s\phi\), the differentials are
\begin{align*}
d_1^*(s\phi)
&=
\left(
 s(X^{-1}\phi-\phi X),
 \,
 s(Y^{-1}\phi-\phi Y)
\right),\\
d_2^*(s\phi_1,s\phi_2)
&=
 s\left(
 Y^{-1}\phi_1-q_0\phi_1Y-q_0X^{-1}\phi_2+\phi_2X
\right).
\end{align*}
Hence
\[
HH^2(B_{q_0},B_{q_0}s)
=
B_{q_0}s/\operatorname{Im}(d_2^*).
\]
Applying \(d_2^*\) to Laurent monomials shows that the exponents of \(X\) and \(Y\) may be reduced modulo \(2\); in particular,
\[
[sX^{-1}]=q_0^{-1}[sX],
\qquad
[sY^{-1}]=q_0[sY].
\]
Hence
\[
HH^2(B_{q_0},B_{q_0}s)
=
\operatorname{Span}\{[s],[sX],[sY],[sXY]\}
\simeq\Cb^4.
\]
The nontrivial element of \(W_0(C_1)\) also acts on the top Koszul wedge, and the resulting scalar factors cancel the above powers of \(q_0\). Thus all four classes are \(W_0(C_1)\)-invariant, so
\[
HH^2(B_{q_0},B_{q_0}s)^{W_0(C_1)}
=
\operatorname{Span}\{[s],[sX],[sY],[sXY]\}.
\]
Finally, we found the coefficients of each of those basis elements for $\mu_\beta(Y,X)$ is a nonzero scalar multiple of one from each of $u_0',t_0',u_1',t_1'$. This completes the proof.
\end{proof}

Finally, we compute $\text{pr}_{[s_1]}([\mu_{\beta_{\underline{t}}}])$.
As $B_{q_0}$-bimodules,
$$B_{q_0}[n]s_1=B_{q_0}[2]s_1\otimes B_{q_0}[n-2].$$
\noindent Now, $B_{q_0}[2]=B_{q_0}[1]^{\otimes 2}$ and $s_1$ acts by swapping the factors. Thus by a lemma on cyclic permutation \cite[Prop. 2.1]{EO} and Van den Bergh duality \cite{VanDenBergh1998},
\[HH^i(B_{q_0}[2],B{q_0}[2]s_1)=HH_{4-i}(B_{q_0}[2],B{q_0}[2]s_1)=HH_{4-i}(B_{q_0}[1])=HH^{i-2}(B_{q_0}[1])\]
In particular, cohomology vanishes for $i<2$ and we conclude by Kunneth theorem
$$HH^2(B_{q_0}[n],B_{q_0}[n]s_1)=HH^2(B_{q_0}[2],B_{q_0}[2]s_1)$$
This shows we are reduced to a rank 2 computation.

\begin{lemma}\label{s_1 projection}
    For 
    $$(q,\underline{\kappa})=(q_0,1,1,1,1,1+\ep t')$$
    defining the rank 2 DAHA $H_{q,\underline{\kappa}}(CC_2^\vee)$,
    the cocycle $\mu_\beta$ as defined in \ref{hochschild cocycle beta}, takes the following value
    \begin{equation}
        \text{pr}_{s_1}(\mu_\beta(Y_1,X_2))=-t'X_2Y_2[s_1].
    \end{equation}
In particular, the restriction of the Kodaira--Spencer map induces an isomorphism
\[
\Cb t'
\xrightarrow{\ \sim\ }
HH^2(B_q[2],B_q[2]s_1)^{Z(s_1)}.
\]
\end{lemma}
\begin{proof}
Write \(\beta_{Y_1}=Y_1+\ep b_{Y_1}.\)
Since \(t=1+\ep t'\), we have $t^{\frac12}=1+\frac{\ep t'}2
\mod\ep^2.$
Expanding the product defining \(\beta_{Y_1}\) using
\eqref{beta_i basis}, and extracting the $s_1$ component gives
\begin{equation}\label{s1 component bY1}
\operatorname{pr}_{s_1}(b_{Y_1})
=
\frac{t'X_2Y_2}{X_2-X_1}[s_1].
\end{equation}

By Lemma~\ref{mu_beta def new}, and since \(Y_1X_2=X_2Y_1\), we have
\[
\mu_\beta(Y_1,X_2)
=
b_{Y_1}X_2-X_2b_{Y_1}.
\]
Taking the \(s_1\)-component and using $[s_1]X_2=X_1[s_1],$
we obtain
\begin{align*}
\operatorname{pr}_{s_1}
\bigl(\mu_\beta(Y_1,X_2)\bigr)
&=
\frac{t'X_2Y_2}{X_2-X_1}
\bigl([s_1]X_2-X_2[s_1]\bigr)\\
&=
\frac{t'X_2Y_2}{X_2-X_1}
(X_1-X_2)[s_1]\\
&=
-t'X_2Y_2[s_1].
\end{align*}
    Next, we claim that for $t'\neq0$ the $s_1$-component of $\mu_\beta$ is not a coboundary. Indeed, define a linear functional $\Lambda:B_{q_0}[2]\to\Cb$
by specifying it on monomials
\[
\Lambda(X_1^aX_2^bY_1^cY_2^d)
=
\begin{cases}
q_0^{-ac},& a+b=1,\ c+d=1,\\
0,&\text{otherwise}.
\end{cases}
\]
Suppose $\varphi:V\longrightarrow B_{q_0}[2]s_1$
is a Koszul $1$-cochain and write
\[
\varphi(e_{Y_1})=A[s_1],
\qquad
\varphi(e_{X_2})=B[s_1].
\]
By Equation~\ref{Koszul differential}, and since $X_2Y_1=Y_1X_2,
s_1(X_2)=X_1,
s_1(Y_1)=Y_2,$
we have
\[
(d^1\varphi)(e_{Y_1}\wedge e_{X_2})
=
-\bigl(X_2A-AX_1-Y_1B+BY_2\bigr)[s_1].
\]
A direct check on monomials gives
\[
\Lambda(X_2A)=\Lambda(AX_1),
\qquad
\Lambda(Y_1B)=\Lambda(BY_2).
\]
Thus, $\Lambda$ vanishes on the
$e_{Y_1}\wedge e_{X_2}$ component of every coboundary. On the other hand, $\Lambda(X_2Y_2)=1.$
    Therefore, the $s_1$-component of $\mu_\beta$ is not a coboundary. Thus $\operatorname{pr}_{[s_1]}([\mu_\beta])$ is a nonzero class in $HH^2(B_{q_0}[2],B_{q_0}[2]s_1)^{Z(s_1)}\simeq\Cb$,
and the $t'$-parameter maps isomorphically onto it.
\end{proof}

Combining Lemmas~\ref{id projection}, \ref{P_1 projection}, and
\ref{s_1 projection}, the Kodaira--Spencer map
\[
KS:(\mathfrak m/\mathfrak m^2)^*
\longrightarrow
HH^2(B_{q_0}[n]\rtimes W_0)
\]
has rank at least \(6\). Indeed, the \(q\) parameter maps nontrivially to the
identity component, the four parameters
\(t_0',u_0',t_n',u_n'\) map isomorphically onto the
\(P_1\)-component, and the \(t'\)-parameter maps nontrivially to the
\(s_1\)-component. These components belong to distinct summands in the
conjugacy-class decomposition
\[
HH^2(B_{q_0}[n]\rtimes W_0)
\simeq
\bigoplus_{[w]\subset W_0}
HH^2(B_{q_0}[n],B_{q_0}[n]w)^{Z(w)},
\]
and hence the corresponding six classes are linearly independent.
By \cite{EO}, the target has dimension exactly 
\(6\). Therefore
\[
KS:(\mathfrak m/\mathfrak m^2)^*
\xrightarrow{\sim}
HH^2(B_{q_0}[n]\rtimes W_0)
\]
is an isomorphism. This proves Theorem~\ref{isom of HH^2}, and hence
Theorem~\ref{DAHA universal deformation}.

\bibliographystyle{alpha}
\bibliography{references}

@incollection{EO,
  author    = {Etingof, Pavel and Oblomkov, Alexei},
  title     = {Quantization, orbifold cohomology, and {C}herednik algebras},
  booktitle = {Jack, Hall-Littlewood and Macdonald Polynomials},
  series    = {Contemporary Mathematics},
  volume    = {417},
  pages     = {171--182},
  publisher = {American Mathematical Society},
  address   = {Providence, RI},
  year      = {2006},
  eprint    = {math/0311005},
  archivePrefix = {arXiv}
}

@misc{Schedler2012,
author        = {Schedler, Travis},
title         = {Deformations of algebras in noncommutative geometry},
year          = {2012},
eprint        = {1212.0914},
archivePrefix = {arXiv},
primaryClass  = {math.QA}
}

@article{Noumi1995,
author  = {Noumi, Masatoshi},
title   = {Macdonald–Koornwinder polynomials and affine Hecke rings},
journal = {RIMS Kokyuroku},
volume  = {919},
pages   = {44–55},
year    = {1995},
note    = {In Japanese}
}

@misc{CrisanConvolutionAlgebras,
author        = {Cri{\c{s}}an, Drago{\c{s}}},
title         = {Convolution Algebras Associated to Representations},
eprint        = {2606.19783},
archivePrefix = {arXiv},
primaryClass  = {math.RT},
year          = {2026}
}

@misc{YoshidaQuantizedCoulombDAHA,
author        = {Yoshida, Yutaka},
title         = {Quantized {C}oulomb Branch of {$4d$} {$\mathcal{N}=2$} {$Sp(N)$} Gauge Theory and Spherical {DAHA} of {$(C_N^\vee,C_N)$}-Type},
eprint        = {2503.15446},
archivePrefix = {arXiv},
primaryClass  = {hep-th},
year          = {2025}
}

@misc{BaranovskyEvensGinzburgQuantumToriDAHA,
author        = {Baranovsky, Vladimir and Evens, Sam and Ginzburg, Victor},
title         = {Representations of Quantum Tori and Double-Affine {H}ecke Algebras},
eprint        = {math/0005024},
archivePrefix = {arXiv},
primaryClass  = {math.RT},
year          = {2000}
}

@article{GKV,
  author  = {Ginzburg, Victor and Kapranov, Mikhail and Vasserot, Eric},
  title   = {Residue construction of {H}ecke algebras},
  journal = {Advances in Mathematics},
  volume  = {128},
  number  = {1},
  pages   = {1--19},
  year    = {1997},
  eprint  = {alg-geom/9512017},
  archivePrefix = {arXiv}
}

@article{Sahi1999,
  author  = {Sahi, Siddhartha},
  title   = {Nonsymmetric {K}oornwinder polynomials and duality},
  journal = {Annals of Mathematics},
  series  = {2},
  volume  = {150},
  number  = {1},
  pages   = {267--282},
  year    = {1999},
  doi     = {10.2307/121102}
}

@article{Stokman2000,
  author  = {Stokman, Jasper V.},
  title   = {{K}oornwinder polynomials and affine {H}ecke algebras},
  journal = {International Mathematics Research Notices},
  volume  = {2000},
  number  = {19},
  pages   = {1005--1042},
  year    = {2000},
  eprint  = {math/0002090},
  archivePrefix = {arXiv}
}

@article{Oblomkov2004,
  author  = {Oblomkov, Alexei},
  title   = {Double affine {H}ecke algebras of rank {$1$} and affine cubic surfaces},
  journal = {International Mathematics Research Notices},
  volume  = {2004},
  number  = {18},
  pages   = {877--912},
  year    = {2004},
  eprint  = {math/0306393},
  archivePrefix = {arXiv}
}

@article{SheplerWitherspoonLie,
  author  = {Shepler, Anne V. and Witherspoon, Sarah},
  title   = {Group actions on algebras and the graded {L}ie structure of {H}ochschild cohomology},
  journal = {Journal of Algebra},
  volume  = {351},
  pages   = {350--381},
  year    = {2012},
  eprint  = {0911.0938},
  archivePrefix = {arXiv}
}

@article{Stefan1995,
  author  = {{\c S}tefan, Drago{\c s}},
  title   = {{H}ochschild cohomology on {H}opf {G}alois extensions},
  journal = {Journal of Pure and Applied Algebra},
  volume  = {103},
  number  = {2},
  pages   = {221--233},
  year    = {1995}
}

@book{WitherspoonBook,
  author    = {Witherspoon, Sarah J.},
  title     = {{H}ochschild Cohomology for Algebras},
  series    = {Graduate Studies in Mathematics},
  volume    = {204},
  publisher = {American Mathematical Society},
  address   = {Providence, RI},
  year      = {2019}
}

@article{EtingofCherednikHeckeVarieties,
  author  = {Etingof, Pavel},
  title   = {{C}herednik and {H}ecke algebras of varieties with a finite group action},
  journal = {Moscow Mathematical Journal},
  volume  = {17},
  number  = {4},
  year    = {2017},
  pages   = {635--666},
  eprint  = {math/0406499},
  archivePrefix = {arXiv},
  primaryClass  = {math.QA}
}

@phdthesis{Vitanov2019,
  author  = {Vitanov, Alexander},
  title   = {On the Deformation Theory of Sheaves of Noncommutative Associative Algebras},
  school  = {ETH Z{\"u}rich},
  year    = {2019},
  type    = {Doctoral thesis},
  number  = {25824},
  doi     = {10.3929/ethz-b-000369120}
}

@article{VanDenBergh1998,
  author  = {Van den Bergh, Michel},
  title   = {A Relation Between {H}ochschild Homology and Cohomology for {G}orenstein Rings},
  journal = {Proceedings of the American Mathematical Society},
  volume  = {126},
  number  = {5},
  pages   = {1345--1348},
  year    = {1998},
  doi     = {10.1090/S0002-9939-98-04210-5}
}

@unpublished{SchraderShapiroCoulombResidues,
  author = {Schrader, Gus and Shapiro, Alexander},
  title  = {Cluster Structures for {$K$}-Theoretic {C}oulomb Branches of Quiver Theories via Residues},
  note   = {Work in preparation; announced in a lecture at the Perimeter Institute, March 2026},
  year   = {2026}
}

@unpublished{KlyuevResidueCoulombBranches,
  author = {Klyuev, Daniil},
  title  = {Residue Construction of Quantized {C}oulomb Branches},
  note   = {In preparation}
}

@article{WitherspoonZhouGerstenhaber,
  author  = {Witherspoon, Sarah and Zhou, Guodong},
  title   = {{G}erstenhaber Brackets on {H}ochschild Cohomology of Quantum Symmetric Algebras and Their Group Extensions},
  journal = {Pacific Journal of Mathematics},
  volume  = {283},
  number  = {1},
  year    = {2016},
  pages   = {223--255},
  doi     = {10.2140/pjm.2016.283.223}
}

@article{WambstQuantumTorus,
  author  = {Wambst, Marc},
  title   = {{H}ochschild and Cyclic Homology of the Quantum Multiparametric Torus},
  journal = {Journal of Pure and Applied Algebra},
  volume  = {114},
  number  = {3},
  pages   = {321--329},
  year    = {1997},
  doi     = {10.1016/S0022-4049(95)00169-7}
}

@article{Yamaguchi2022,
  author  = {Yamaguchi, Kohei},
  title   = {A {L}ittlewood--{R}ichardson Rule for {K}oornwinder Polynomials},
  journal = {Journal of Algebraic Combinatorics},
  volume  = {56},
  pages   = {335--381},
  year    = {2022},
  doi     = {10.1007/s10801-022-01114-5}
}

@incollection{Lusztig1989,
  author    = {Lusztig, George},
  title     = {Affine {H}ecke Algebras and Their Graded Version},
  booktitle = {Algebraic Groups and Related Topics},
  series    = {Advanced Studies in Pure Mathematics},
  volume    = {6},
  pages     = {97--117},
  publisher = {Kinokuniya},
  address   = {Tokyo},
  year      = {1985}
}

@article{BFN,
  author  = {Braverman, Alexander and Finkelberg, Michael and Nakajima, Hiraku},
  title   = {Towards a mathematical definition of {C}oulomb branches of {$3$}-dimensional {$\mathcal N=4$} gauge theories, {II}},
  journal = {Advances in Theoretical and Mathematical Physics},
  volume  = {22},
  number  = {1},
  pages   = {107--146},
  year    = {2018},
  eprint  = {1601.03586},
  archivePrefix = {arXiv},
  primaryClass  = {math.RT}
}
\end{document}